\documentclass[11pt,a4paper]{article}
\usepackage{amsmath}
\usepackage{amssymb}
\usepackage{amsthm}
\usepackage{amsfonts}
\usepackage{latexsym}
\usepackage{verbatim} 
\usepackage{xcolor}
\usepackage[a4paper,width=16cm,top=2cm,bottom=2cm,footskip=1cm
]{geometry}

\usepackage[flushmargin]{footmisc}
\usepackage{color}                    
\usepackage{esvect}
\usepackage{mathrsfs}

\theoremstyle{plain}
\newtheorem{thm}{Theorem}[section]
\newtheorem{cor}{Corollary}[section]
\newtheorem{lem}{Lemma}[section]
\newtheorem{prop}{Proposition}[section]

\newtheorem{ass}{Assumption}[section]
\theoremstyle{remark}

\theoremstyle{definition}
\newtheorem{defn}{Definition}[section]

\newtheorem{exmp}{Example}[section]

\newtheorem{rem}{Remark}[section]

\newcommand{\Complex}{\mathbb C}
\newcommand{\Real}{\mathbb R}
\newcommand{\N}{\mathbb N}
\newcommand{\ddbar}{\overline\partial}
\newcommand{\pr}{\partial}
\newcommand{\ol}{\overline}

\newcommand{\norm}[1]{\left\Vert#1\right\Vert}
\newcommand{\abs}[1]{\left\vert#1\right\vert}
\newcommand{\set}[1]{\left\{#1\right\}}
\newcommand{\To}{\rightarrow}
\DeclareMathOperator{\Ker}{Ker}

\title{On the singularities of the Szeg\H{o} kernels on CR orbifolds}
\author{Andrea Galasso and Chin-Yu Hsiao\footnote{\noindent{\bf Address:} Room 407, Chee-Chun Leung Cosmology Hall, National Taiwan University; {\bf ORCID iD:} 0000-0002-5792-1674; {\bf e-mail}: andrea.galasso@ncts.ntu.edu.tw {\bf Address:} Institute of Mathematics, Academia Sinica, 6F, Astronomy-Mathematics Building, No.1, Sec.4, Roosevel: {\bf ORCHID iD:} 0000-0002-1781-0013; {\bf email}: chsiao@math.sinica.edu.tw; chinyu.hsiao@gmail.com} }

\date{}

\begin{document}
	\maketitle

	\begin{abstract} In this paper we study the microlocal properties of the Szeg\H{o} kernel of a given compact connected orientable CR orbifold whose Kohn Laplacian has closed range. This last assumption is satisfied if certain geometric conditions hold true, as in the smooth case. 
		As applications, we give a pure analytic proof of Kodaira-Bailey theorem and explain how to generalize a CR version of quantization commutes with reduction to orbifolds.
	\end{abstract}
	\tableofcontents
	
	
	
	\section{Introduction}
	
	The study of Szeg\H{o} kernels on CR manifolds play an important role in CR/Complex analysis and geometry. Boutet de Monvel and Sj\"ostrand~\cite{bs} showed that the Szeg\H{o} kernel on a strongly pseudoconex CR manifold is a complex Fourier integral operator. Boutet de Monvel and Sj\"ostrand's result has a profound impact in CR/Complex analysis and geometry. 
	When $X$ is a CR orbifold, the study of the associated Szeg\H{o} kernel on $X$ is a very natural question and could have some application in CR/Complex geometry and geometric quantization theory. For example, in~\cite{HMM} and~\cite{hsiaohuang}, the authors studied quantization commutes with reduction for stronly speudococonex CR manifolds under the assumption that the group action $G$ is free near $\mu^{-1}(0)$, $\mu$ is the associated momentum map. When the group action $G$ is locally free near $\mu^{-1}(0)$, we need to understand the behaviour of the Szeg\H{o} kernels on CR orbifolds. The main goal of this paper is to study the associated Szeg\H{o} kernel on a CR orbifold $X$ with non-degenerate Levi form. As applications, we give a pure analytic proof of Kodaira-Bailey theorem and explain how to generalize a CR version of quantization commutes with reduction to orbifolds.

	We now formulate our main results. We refer the reader to Section~\ref{sec: orbifold} for some notations and terminology used here. 
	Let $(X,\,T^{1,0}X)$ be a compact orientable connected CR orbifold of dimension $2n+1,\, n\geq 1$. In Section \ref{sec: orbifold} we recall the basic material concerning orbifolds, we give definitions for the spaces of $(0,q)$-forms with $L^2(X)$ coefficients and for the Kohn Laplacian $\square_b^{(q)}$ on the orbifold $X$. We can then define abstractly the Szeg\H{o} projector for $X$ to be the projector
	\[ \Pi^{(q)}\,:\, L^2_{(0,q)}(X)\rightarrow \Ker \square_b^{(q)} \,. \]
	
	The main aim of this paper is to prove the following result, which is a generalization of Theorem $1.2$ in \cite{hsiao}. 
	
	\begin{thm} \label{thm:main}
		Assume that the Levi form is non-degenerate of constant signature $(n_-,n_+)$ on $X$. 
		Suppose $\square_b^{(q)}$ has closed range.  Then $$\Pi^{(q)}\equiv \Pi_-+\Pi_+$$ where $\Pi_{\pm} \equiv 0$ if $q\neq n_{\pm}$. Let $q=n_-$. We have $\mathrm{WF}'(\Pi_{-})=\mathrm{diag}(\Sigma^{-}\times \Sigma^{-})$, 
		$\mathrm{WF}'(\Pi_{+})=\mathrm{diag}(\Sigma^{+}\times \Sigma^{+})$ if $q=n_-=n_+$,
		\[\Sigma^{\pm}= \left\{(x,\,\lambda\,\omega_0(x))\in T^*X\,:\, \lambda\, _{<}^{>} \, 0 \right\}\,,  \] 
		$\mathrm{WF}'(\Pi_{\pm})=\set{(x,\xi,y,-\eta)\in T^*X\times T^*X;\, (x,\xi,y,\eta)\in\mathrm{WF}(\Pi_{\pm})}$, $\mathrm{WF}(\Pi_{\pm})$ denotes the wave front set of $\Pi_{\pm}$ in the sense of H\"ormander. Moreover, $\Pi_{\pm}$ has the following microlocal expression: Let $\Pi_{\pm}(x,y)\in\mathcal{D}'(X\times X,T^{*0,q}X\boxtimes(T^{*0,q}X)^*)$ be the distribution kernel of $\Pi_{\pm}$. Consider an open set $U\subset X$ and an orbifold chart $(\widetilde{U},G_U)\rightarrow U$.
		We have 
		\[\Pi_{\pm}(x,y)=\sum_{g\in G_U}\Bigr(\int _0^{+\infty}e^{it\, \varphi_{\pm}(\widetilde{x},g\cdot \widetilde{y})} s_{\pm}(\widetilde{x},g\cdot \widetilde{y},t)\,\mathrm{d}t+F_{\pm}(\widetilde x,g\cdot\widetilde y)\Bigr)\ \ \mbox{on $U\times U$} \,,\]
		where $\pi(\widetilde x)=x$, $\pi(\widetilde y)=y$, $\pi: \widetilde U\To U$ is the natural projection, $F_{\pm}$ are smoothing operators on $\widetilde U$ and 
		\[s_{\pm}( \widetilde{x}, \widetilde{y},t)\sim\sum^\infty_{j=0}s^j_{\pm}( \widetilde{x},  \widetilde{y})t^{n-j}\]
		in $S^{n}_{1, 0}( \widetilde{U}\times  \widetilde{U}\times\mathbb{R}_+\,,T^{*0,q}X\boxtimes(T^{*0,q}X)^*)$,
		\[\mbox{$s_{+}( \widetilde{x}, \widetilde{y},t)=0$ if $q\neq n_+$}\]
		and
		\[s^0_{-}(\widetilde{x}_0, \widetilde{x}_0)=\frac{1}{2}\pi^{-n-1}\lvert\det\mathcal{L}_{ \widetilde{x}_0}\rvert\tau_{ \widetilde{x}_0,n_{-}}\,, \widetilde{x}_0\in  \widetilde{U},\]
		if $q=n_-=n_+$, 
		\[s^0_{+}(\widetilde{x}_0, \widetilde{x}_0)=\frac{1}{2}\pi^{-n-1}\lvert\det\mathcal{L}_{ \widetilde{x}_0}\rvert\tau_{ \widetilde{x}_0,n_{+}}\,, \widetilde{x}_0\in  \widetilde{U},\]
		and the phase functions $\varphi_-$, $\varphi_+$ satisfy
		\[
		\renewcommand{\arraystretch}{1.2}
		\begin{array}{ll}
			&\varphi_+, \varphi_-\in\mathcal{C}^\infty( \widetilde{U}\times  \widetilde{U}),\ \ {\rm Im\,}\varphi_{\pm}( \widetilde{x},  \widetilde{y})\geq0,\\
			&\varphi_-( \widetilde{x},  \widetilde{x})=0,\ \ \varphi_-( \widetilde{x}, \widetilde{y})\neq0\ \ \mbox{if}\ \  \widetilde{x}\neq\widetilde{y},\\
			&\mathrm{d}_x\varphi_-( \widetilde{x},  \widetilde{y})\big|_{ \widetilde{x}= \widetilde{y}}=-\omega_0( \widetilde{x}), \ \ \mathrm{d}_{ \widetilde{y}}\varphi_-( \widetilde{x},  \widetilde{y})\big|_{ \widetilde{x}= \widetilde{y}}=\omega_0( \widetilde{x}), \\
			&-\ol\varphi_+( \widetilde{x},  \widetilde{y})=\varphi_-( \widetilde{x}, \widetilde{y}).
		\end{array}
		\] 
		
		Moreover, we have 
		\begin{equation}\label{e-gue220731yyd}
			\begin{split}
				&\sum_{g\in G_U}\Bigr(\int _0^{+\infty}e^{it\, \varphi_{\pm}(\widetilde{x},g\cdot \widetilde{y})} s_{\pm}(\widetilde{x},g\cdot \widetilde{y},t)\,\mathrm{d}t+F_{\pm}(\widetilde x,g\cdot\widetilde y)\Bigr)\\
				&=\sum_{h,g\in G_U}\frac{1}{\abs{G_U}}\Bigr(\int _0^{+\infty}e^{it\, \varphi_{\pm}(h\circ\widetilde{x},g\cdot \widetilde{y})} s_{\pm}(h\cdot\widetilde{x},g\cdot \widetilde{y},t)\,\mathrm{d}t+F_{\pm}(h\circ\widetilde x,g\cdot\widetilde y)\Bigr),
			\end{split}
		\end{equation}
		where $\mathcal{L}_{\widetilde x_0}$ denotes the Levi form of $X$ at $\widetilde x_0$, 
		${\rm det\,}\mathcal{L}_{\widetilde x_0}=\lambda_1(\widetilde x_0)\cdots\lambda_n(\widetilde x_0)$, $\lambda_j(\widetilde x_0)$, $j=1,\ldots,n$, are the eigenvalues of $\mathcal{L}_{\widetilde x_0}$ with respect to the given Hermitian metric on $\mathbb CTX$ and 
		$\tau_{ \widetilde{x}_0,n_{\pm}}$ are given by \eqref{tau140530}. 
	\end{thm}
	
	The main idea of the proof is the following: On $\widetilde U$, by using approximate Szeg\H{o} kernel $S$ constructed in~\cite{hsiao} to define a kernel $$S^{(q)}_U(x,y):=\sum_{g,h\in G_U}\frac{1}{\abs{G_U}}S(h\cdot\widetilde x,g\cdot\widetilde y)$$ which is micro-locally given by the formulas above. We then show that $S^{(q)}_U$ is microlocally self-adjoint and it behaves as the identity when restricted to $\Ker \square_b^{(q)}$, up to smoothing contributions. Moreover, by using the result in~\cite{HM} about Szeg\H{o} kernel asymptotiocs for lower energy forms, we can show that $S^{(q)}_U$ satisfies \eqref{e-gue220731yyd}. 
	Then using $S^{(q)}_U$ one can define globally an operator $S^{(q)}$ on $X$ which we prove to be the projector $\Pi^{(q)}$ up to some smoothing operator, if $\square_b^{(q)}$ has closed range. This is a consequence of the properties of $S^{(q)}_U$ just mentioned and a theorem of functional analysis. The property \eqref{e-gue220731yyd} is important for further study about the algebra of Toeplitz operators on CR orbifolds.

	Now, assume that $X$ admits a transversal and CR locally free $S^1$ action $e^{i\theta}$. We take the Reeb vector field $R$ to be the vector field on $X$ induced by the $S^1$ action. For every $\ell\in\mathbb N$, put 
	\[X_\ell:=\set{x\in X;\, e^{i\theta}x\neq x, \theta\in\left[0,\frac{2\pi}{\ell}\right[, e^{i\frac{2\pi}{\ell}}x=x}.\]
	For every $k\in\mathbb Z$, set 
	\[{\rm Ker\,}\Box^{(q)}_{b,k}:=\set{u\in{\rm Ker\,}\Box^{(q)}_b;\, (e^{i\theta})^*u=e^{ik\theta}u}.\]
	Let 
	\[\Pi^{(q)}_k: L^2_{(0,q)}(X)\To{\rm Ker\,}\Box^{(q)}_{b,k}\]
	be the orthogonal projection and let $\Pi^{(q)}_k(x,y)\in\mathcal{C}^\infty(X\times X,T^{*0,q}X\boxtimes(T^{*0,q}X)^*)$ be the distribution kernel of $\Pi^{(q)}_k$.
	From Theorem~\ref{thm:main}, we can repeat the procedure in \cite{gh} and get 
	
	\begin{thm}\label{t-gue220731yyd}
		With the assumptions and notations used above, assume that the Levi form is non-degenerate of constant signature $(n_-,n_+)$ on $X$ and let $q=n_-$. Fix $p\in X$ and assume that $p\in X_\ell$, for some $\ell\in\mathbb N$. Consider an open set $U\subset X$, $p\in U$, and an orbifold chart $(\widetilde{U},G_U)\rightarrow U$. 
		We have as $k\To+\infty$,
		\begin{equation}\label{eq:szegos1free}
			\Pi^{(q)}_k(x,y)=\sum_{g\in G_U}\sum^{\ell-1}_{j=0}e^{\frac{2\pi kj}{\ell}} e^{ik\, \Psi(\widetilde{x},e^{i\frac{2\pi j}{\ell}}\cdot g\cdot\widetilde{y})}b(\widetilde{x},e^{i\frac{2\pi j}{\ell}}\cdot g\cdot \widetilde{y},k)+
			\sum_{g\in G_U}\sum^{\ell-1}_{j=0}F_k(\widetilde x,e^{i\frac{2\pi j}{\ell}}\cdot g\cdot \widetilde{y})
		\end{equation}
		where 
		\begin{equation}\label{e-gue220731ycd}
			\begin{split}
				&\Psi\in\mathcal{C}^\infty(\widetilde U\times\widetilde U),\\
				&\Psi(\widetilde x,\widetilde x)=0,\ \ \mbox{for all $\widetilde x\in\widetilde U$},\\
				&\frac{1}{C}\inf_{e^{i\theta}\in S^1}\set{d^2(\widetilde x,e^{i\theta}\widetilde y)}\leq {\rm Im\,}\Psi(\widetilde x,\widetilde y)\leq C\inf_{e^{i\theta}\in S^1}\set{d^2(\widetilde x,e^{i\theta}\widetilde y)},\\
				&\mbox{$\forall (\widetilde x,\widetilde y)\in\widetilde U\times\widetilde U$, $C>1$ is a constant},
			\end{split}
		\end{equation}
		\begin{equation}\label{e-gue220731ycdI}
			\begin{split}
				&\mbox{$b(\widetilde x,\widetilde y,k)\sim\sum^{+\infty}_{j=0}b_j(\widetilde x,\widetilde y)k^{n-j}$ in
					$S^n(1;\widetilde U\times\widetilde U,T^{*0,q}X\boxtimes(T^{*0,q}X)^*)$},\\
				&b_j(\widetilde x,\widetilde y)\in\mathcal{C}^\infty(\widetilde U\times\widetilde U,T^{*0,q}X\boxtimes(T^{*0,q}X)^*),\ \ j=0,1,\ldots,\\
				&b_0(\widetilde x,\widetilde x)=\frac{1}{2}\pi^{-n-1}\lvert\det\mathcal{L}_{ \widetilde{x}_0}\rvert\tau_{ \widetilde{x}_0,n_{-}}\,, \widetilde{x}_0\in  \widetilde{U},
			\end{split}
		\end{equation}
		$F_k\in\mathcal{C}^\infty(\widetilde U\times\widetilde U,T^{*0,q}X\boxtimes(T^{*0,q}X)^*)$ satisfies the following: for every $N_1, N\in\mathbb N$, there is a constant $C_{N_1,N}$ independent of $k$ such that  $\norm{F_k(\widetilde x,\widetilde y)}_{\mathcal{C}^{N_1}(\widetilde U\times\widetilde U}\leq C_{N_1,N}k^{-N}$, for every $k\gg1$. 
	\end{thm}
	
	We refer the reader to~\cite[Definition 8.1]{HM} for the meaning of semi-classical symbol spaces 
	$S^n(1;\widetilde U\times\widetilde U,T^{*0,q}X\boxtimes(T^{*0,q}X)^*)$.
	
	
	
	As an easy consequence of the asymptotic expansion of the Szeg\H{o} kernel along the diagonal we also have the following theorem, see also Corollary $1.1$ in \cite{p}. 
	
	
	\begin{thm} \label{thm:dim}
		Under the same assumptions and notations as in Theorem~\ref{t-gue220731yyd} above, 
		assume that $X=X_{\ell_0}\cup X_{\ell_1}\cup\cdots\cup X_{\ell_u}$, $1\leq\ell_0<\ell_1<\cdots<\ell_u$, $\ell_j\in\mathbb N$, $j=0,1,\ldots,u$, $u\in\mathbb N$.  Denote by $p$ the least common multiple of $l_0,l_1,\dots, l_n$. 
		We have
		\[\lim_{k\rightarrow +\infty}\frac{1}{(kp)^n}{\rm dim\,}{\rm Ker\,}\Box^{(q)}_{b,kp}=\frac{\ell_0}{2}\pi^{-n-1}\int_X\abs{{\rm det\,}\mathcal{L}_x}\mathrm{dV}_X.\]
	\end{thm} 
	
	We work with the same assumptions and notations as in Theorem~\ref{t-gue220731yyd} above and assume that $X$ is strongly pseudoconvex.
	Let $M$ be the Hodge orbifold obtained by $M=X/S^1$. There exists a complex proper line orbibundle $(L,h^L)$ such that $X\rightarrow M$ is the circle orbibundle $X\subset L^*$.
	For $k\in\mathbb N$, let $H^0(M,L^k)$ be the space of global holomorphic sections of $M$ with values in $L^k$. The space ${\rm Ker\,}\Box^{(0)}_{b,k}$ is isomorphic to $H^0(M,L^k)$. 
	Let $\set{f_j}^{d_k}_{j=1}\subset H^0(M,L^k)$ be an orthonormal basis. 
	The Kodaira-Bailey map is given by 
	\begin{equation}\label{e-gue220805yyd}
		\begin{split}
			\Phi_{k}:M&\To \mathbb{C}\mathbb P^{d_k-1},\\
			x&\To[f_1(x),\ldots,f_{d_k}(x)].
		\end{split}
	\end{equation}
	As an application of Theorem~\ref{t-gue220731yyd}, we will give an analytic proof of the following Kodaira-Bailey embedding theorem
	
	\begin{thm}\label{thm:genkodaira} 
		With the same assumptions and notations as in Theorem~\ref{t-gue220731yyd} and suppose that $X$ is strongly pseudoconvex. Then, the Kodaira-Bailey map
		$\Phi_{kp}:M \hookrightarrow \mathbb{C}\mathbb P^{d_{kp}-1}$ is an embedding for $k$ sufficiently large, where $p$ is the least common multiple of $l_0,l_1,\dots, l_n$. 
	\end{thm}
	
	It should be mention that Ma and Marinescu~\cite{mm} also gave a proof of Theorem~\ref{thm:genkodaira}. 
	
	We assume that $X$ admits a CR compact Lie group action $G$, ${\rm dim\,}G=d$. We do not assume that $X$ admits a transversal and $S^1$ action. 
	We will work with Assumption~\ref{a-gue170123I} below and suppose that $X$ is strongly pseudoconvex. We refer the reader to Section~\ref{s-gue220801yyd} and Section~\ref{sec:app} for the notations and terminology used here.  We will assume that $0$ is a regular value of the moment map $\mu$ and thus $\mu^{-1}(0)$ is a suborbifold of $X$. Put $X_G:=\mu^{-1}(0)/G$. Note that both $X$ and $X_G$ are orbifolds. 
	Let $\Box^{(0)}_{b,X_G}$ be the Kohn Laplacian on the orbifold  $X_G$. Let $$({\rm Ker\,}\Box^{(0)}_b)^G:=\set{u\in{\rm Ker\,}\Box^{(0)}_b;\, g^*u=u, \forall g\in G}\,.$$ As an application of Theorem~\ref{thm:main}, we generalize a CR version of quantization commutes with reduction to orbifolds.
	
	\begin{thm}\label{t-gue22082yyd}
		There is a map $\sigma: ({\rm Ker\,}\Box^{(0)}_b)^G\subset L^2(X)\To{\rm Ker\,}\Box^{(0)}_{b,X_G}\subset L^2(X_G)$ such that $\sigma$ is Fredholm, that is, 
		\begin{itemize}
			\item[(i)] ${\rm Ker\,}\sigma\cap({\rm Ker\,}\Box^{(0)}_{b})^G$ is a finite dimensional subspace of $\mathcal{C}^{\infty}(X)$,
			\item[(ii)] ${\rm Coker\,}\sigma\cap {\rm Ker\,}\Box^{(0)}_{b,X_G}$ is a finite dimensional subspace of $\mathcal{C}^{\infty}(X_G)$.
		\end{itemize}
	\end{thm}
	
	We refer the reader to \eqref{e-gue180308II} for the explicit formula for $\sigma$. 
	
	Let us recall prior literature. First there is the paper \cite{song}, which deals with the Szeg\H{o} kernel on an orbifold circle bundle. The semi-classical behavior of Bergman kernels on complex orbifolds was studied by Dai, Liu and Ma~\cite{dlm}. Our approach is closer in spirit to \cite{dlm} where an explicit expression for an operator on a local chart $(\tilde{U},G_U)\rightarrow U$ is given. This is also motivated by previous results concerning $G$-invariant Szeg\H{o} kernel for locally free $G$-action on CR manifolds we obtained in \cite{gh}, in this case the CR reduction is an orbifold. 
	
	It is known that a compact connected strongly pseudoconvex CR orbifold $X$ with a transversal CR $S^1$ action can be identified with the circle orbibundle of a positive holomorphic orbifold line bundle $(L, h) \rightarrow M$
	over a compact complex Hodge orbifold $M$. By this identification the $k$-th Fourier component of the Szeg\H{o} kernel function and the Bergman kernel function for the $k$-th tensor power of the line bundle $L$ are equal up to a constant factor of $2\pi$. Hence in this paper we obtained a phase function version of the Bergman kernels expansions on orbifolds (see \cite[Theorems 5.4.10, 5.4.11]{mm} or \cite[Theorem 1.4]{dlm}  for another version). 
	
	\section{Preliminaries}
	\label{sec: orbifold}
	
	\subsection{Orbifolds}
	\label{sec:orbi}
	
	We recall basic definitions we need about orbifolds, for a more precise discussion see \cite{mm} and \cite{alr} and references therein. Define $\mathcal{M}_s$ to be the category whose objects are pairs $(M, G)$, where $M$ is a connected smooth manifold and $G$ is a finite group acting smoothly on $M$, and whose morphisms $\Phi\,:\,(M,G)\rightarrow (M',G')$ are families of open embeddings $\varphi:\,M\rightarrow M'$ such that for each $\varphi\in \Phi$ there exists an injective group homomorphism $\lambda_{\varphi}: G\rightarrow G'$ satisfying
	\[\varphi(g\cdot x)=\lambda_{\varphi}(g)\cdot \varphi(x),\qquad x\in M,\,g\in G \,. \]
	Let $X$ be a paracompact Hausdorff space and let $\mathcal{U}$ be a covering of $X$ consisting of connected open subsets. Then we say that the topological space $X$ has an orbifold structure if 
	\begin{itemize}
		\item[i)] for each $U\in \mathcal{U}$ there exists $(\widetilde{U},G_U)\in \mathcal{M}_s$ and a ramified covering $(\widetilde{U},G_U)\rightarrow U$ such that $U\cong \widetilde{U}/G_U$;  
		\item[ii)] for any $U,\,V \in \mathcal{U}$ there exists a morphism $\Phi_{V,U}: (\widetilde{U},G_U) \rightarrow (\widetilde{V},G_V)$ which covers the inclusion  $U\subset V$ and satisfies $\varphi_{WU}= \varphi_{WV}\circ \varphi_{VU}$ for any $U,\,V,\,W \in \mathcal{U}$ such that $U\subset V \subset W$.
	\end{itemize}
	
	We recall that for each $x\in X$ one can always find local coordinates $(\widetilde{U}, G_U)$, $\widetilde{U}\subseteq \mathbb{R}^n$, such that $x$ is a fixed point of $G_U$ and $G_U$ acts linearly on $\mathbb{R}^n$, $x\equiv 0\in \mathbb{R}^n$. If $\vert G_U\rvert > 1$ then we say that $x$ is a singular point, otherwise we say that $x$ is regular. There exists an open and dense subset in $X$ on which all the points have conjugated stabilizers (see \cite{dk} p. 116), thus if the action of $G$ on $M$ in the definition of $\mathcal{M}_s$ is assumed to be effective as in \cite{mm}, then set of regular points is open and dense in $X$. 
	
	
	A \textit{vector orbibundle} $(E,p)$ on an orbifold $X$ is an orbifold $E$ together with a continuous projection $p:E\rightarrow X$ such that for $U\in \mathcal{U}$ there exists a $G_U$-equivariant lift $\widetilde{p}\,: \widetilde{E}_{U}\rightarrow \widetilde{U}$ defining a trivial vector bundle such that the diagram
	\begin{align*}
		(\widetilde{E}_U,&G^E_U) \xrightarrow{\widetilde{p}} (\widetilde{U}, G_U)   \\
		\downarrow& \qquad\qquad \downarrow \\
		E_U& \quad\,\,\,\,\xrightarrow{p}  \quad U
	\end{align*}
	commutes. As before, when we say that $\widetilde{p}\,: \widetilde{E}_{U}\rightarrow \widetilde{U}$ is $G_U$-equivariant we mean that there exists an injective group homomorphism $\lambda_{p}: G^E_U \rightarrow G_U$ such that
	\[\widetilde{p}(g\cdot e)=\lambda_{p}(g)\cdot \widetilde{p}(e)\qquad e\in \widetilde{E}_U,\,g\in G_U^E \,;\]
	if $\lambda_{p}$ is a group isomorphism, we say that the vector orbibundle is \textit{proper}.
	
	A smooth section of a vector orbibundle $(E,p)$ over $X$ is a smooth map $s: X \rightarrow E$ such that $p\circ s = 1_X$, we denote the space of smooth sections as $\mathcal{C}^\infty(X, E)$. We recall that a \textit{smooth map} $f:X\rightarrow Y$ between orbifolds is a continuous map between the underlying topological spaces such that for each $U\in \mathcal{U}$ there exists a local lift $(\widetilde{f}_U, \overline{f}_U)$ such that $\widetilde{f}_U : \widetilde{U}\rightarrow \widetilde{V}$ is smooth, $\overline{f}_U: G_U\rightarrow G_V$ is a group homomorphism, the diagram
	\begin{align*}
		(\widetilde{U},& G_U) \xrightarrow{\widetilde{f}} (\widetilde{V},G_V) \\
		\downarrow& \qquad\qquad \downarrow \\
		U& \quad\,\,\,\,\xrightarrow{f} \quad V
	\end{align*}
	commutes and $\widetilde{f}_U$ is $\overline{f}_U$-equivariant:
	\[\widetilde{f}_U(g\cdot x)= \overline{f}_U(g)\widetilde{f}_U(x)\,,\qquad g\in G_U,\,x\in\widetilde{U}\,. \]

	\begin{prop}
		The tangent orbibundle $TX$ is defined by gluing together the bundles defined over the charts 
		\[T_{\widetilde{x}}\widetilde{U} \rightarrow T_xU:= T_{\widetilde{x}}\widetilde{U}/G_U\,,  \]
		where the action of $G_U$ is induced by differentiating the action on $\widetilde{U}$, $\mathrm{d}g$. It has a natural structure of proper vector orbibundle. The smooth sections of it are called vector fields. 
	\end{prop}
	
	We can then define the \textit{differential} of a smooth map $f : X \rightarrow Y$ to be the smooth map $\mathrm{d}f: TX \rightarrow TY$ such that locally for a given $x\in U$, $f(x)\in V$ we have a commutative diagram
	\begin{align*}
		(T_{\widetilde{x}}\widetilde{U},& G_U) \xrightarrow{\mathrm{d}_{\widetilde{x}}\widetilde{f}} (T_{\widetilde{f}(\widetilde{x})}\widetilde{V},G_V) \\
		\downarrow& \qquad\qquad \,\,\,\downarrow \\
		T_xU& \quad\,\,\,\,\xrightarrow{\mathrm{d}_xf} \,\, T_{f(x)}V
	\end{align*}
	where $\mathrm{d}_{\widetilde{x}}\widetilde{f}$ is $\overline{f}_U$-equivariant. We say that $f$ is an \textit{immersion} at $x \in X$ if $\mathrm{d}f$ is injective, so this means that there is a local lift $\widetilde{f}_U : \widetilde{U} \rightarrow \widetilde{V}$ such that both $\mathrm{d}_{\widetilde{x}}\widetilde{f}$ and $\overline{f}_U$ are injective. 
	
	There are different ways to construct new vector bundles out of old, see \cite{mist} p. $29$ and also \cite{alr} p. $14$. Given any continuous functor $\mathcal{F}$ from vector spaces to vector spaces we obtain a new orbibundle $\mathcal{F}(TX)\rightarrow X$ with fibers $\mathcal{F}(T_{\widetilde{x}}\widetilde{U})/G_U$. In particular, this allows us to construct the cotangent orbibundle $T^*X$ and exterior powers $\wedge^q T^*X$. We recall the definition of sub-orbibundle that we will require later to define $T^{1,0}X$. A \textit{sub-orbibundle} $(F,p_{\vert F})$ of a vector orbibundle $(E,p)$ is a suborbifold of $E$ being a vector bundle on $X$. We recall that a orbifold $Y$ is a \textit{suborbifold} of an orbifold $X$ if there exists an immersion $\iota : Y \rightarrow X$, we identify $Y$ with its image in $X$. 
	
	\begin{exmp}
		Consider an action of a finite group $G$ on $\mathbb{R}^{2n+1}$. A form on the orbifold $\mathbb{R}^{2n+1}/G$ is a $G$-invariant form on $\mathbb{R}^{2n+1}$.
	\end{exmp}
	
	At this point we can give the definition of exterior derivative. Consider a chart $(\widetilde{U},G_U)\rightarrow U$ of an orbifold $X$, a $k$-form $\omega$ defines a $G_U$-invariant $k$-form $\omega_U$ on $\widetilde{U}$, locally this is given by 
	\[ \omega_U:= f_U\,\mathrm{d}x^{j_1}\wedge \mathrm{d}x^{j_2}\wedge \dots \wedge \mathrm{d}x^{j_k}\qquad \text{ on } \widetilde{U}\subseteq \mathbb{R}^n\,,  \]
	where $f_U$ is a $G_U$-invariant function on $\widetilde{U}$. The exterior derivative of $\omega$ is given locally by
	\[ \mathrm{d}\omega_U:= \sum_{i=1}^n \frac{\partial f_U}{\partial x^i} \mathrm{d}x^i\wedge \mathrm{d}x^J\text{ on } \widetilde{U}\subseteq \mathbb{R}^n\,, \]
	where $J$ is the multi-index $J=(j_1,\dots j_n)$. It is easy to check that it defines a $(k+1)$-form which is invariant under the action of $G_U$. 
	
	If $X$ is an oriented $n$ dimensional orbifold, we define the integral of a smooth section $\omega$ of the orbibundle $\Lambda^n T^*X$ as follows. Suppose that $\mathrm{supp}(\omega)$ is compactly supported on $U \in \mathcal{U}$, we define 
	\[\int_U \omega := \frac{1}{\lvert G_U\rvert}\int_{\widetilde{U}} \widetilde{\omega}_U\,. \]
	In the general case, for every open cover of $X$ by positively oriented charts, $(U_i, G_{U_i})$, we have a partition of unity, $(\rho_i)_{i\in I}$ subordinate to this cover. We define,
	\begin{equation}\label{eq:integral}\int_X \omega := \sum_{i\in I}\frac{1}{\lvert G_{U_i}\rvert}\int_{\widetilde{U}_i} \widetilde{\rho}_i\widetilde{\omega}_{U_i}\,. 
	\end{equation}
	It remains to show that it does not depend on the choice of open cover and on the choice of partition of unity. Now, given another open cover $((V_j,G_{V_j}))_{j\in J}$ and an associated partition of unity $(\sigma_j)_{j\in J}$, since $\mathrm{supp}(\rho_i\sigma_j\omega)\subseteq U_i\cap V_j$ we have 
	\[\frac{1}{\lvert G_{U_i}\rvert}\int_{\widetilde{U}_i} \widetilde{\rho}_i\widetilde{\sigma}_{j,i}\widetilde{\omega}_{U_i} = \frac{1}{\lvert G_{V_j}\rvert}\int_{\widetilde{V}_j} \widetilde{\rho}_{i,j}\widetilde{\sigma}_j\omega_{V_j} \]
	where $\widetilde{\sigma}_{j,i}$ is defined on $\widetilde{U}_i$ as the composition between $\sigma_{j}: U_i \rightarrow \mathbb{R}$ and the covering $\widetilde{U}_i\rightarrow U_i$, in a similar way $\widetilde{\rho}_{i,j}$ is defined. Now,
	\begin{align*}\sum_{i\in I}\frac{1}{\lvert G_{U_i}\rvert}\int_{\widetilde{U}_i} \widetilde{\rho}_i\widetilde{\omega}_{U_i} &=\sum_{i\in I}\sum_{j\in J}\frac{1}{\lvert G_{U_i}\rvert}\int_{\widetilde{U}_i} \widetilde{\rho}_i\widetilde{\sigma}_{j,i}\widetilde{\omega}_{U_i} \\
		& = \sum_{i\in I}\sum_{j\in J} \frac{1}{\lvert G_{V_j}\rvert}\int_{\widetilde{V}_j} \widetilde{\rho}_{i,j}\widetilde{\sigma}_j\omega_{V_j} \\
		& = \sum_{j\in J}\frac{1}{\lvert G_{V_j}\rvert}\int_{\widetilde{V}_j} \widetilde{\sigma}_j\widetilde{\omega}_{V_j}\,,
	\end{align*}
	where we use the fact that 
	\[ \sum_{j\in J}\widetilde{\sigma}_{j,i} = 1,\quad  \sum_{i\in I} \widetilde{\rho}_{i,j}=1\,. \]
	
	\subsection{CR orbifolds}
	
	In this section, we recall the definition of CR orbifolds. We also refer the reader to~\cite[Section 5]{khh}. 
	Let $X$ be a compact, connected and orientable orbifold of dimension $2n+1$, $n\geq 1$, a CR structure on $X$ is a sub-orbibundle $T^{1,0}X$ of rank $n$ of the complexified orbifold tangent bundle $TX\otimes \mathbb{C}$ (which is defined locally for each $x\in X$ by taking $T_{\tilde{x}}\widetilde{U}\otimes \mathbb{C}$), satisfying $T^{1,0}X\cap T^{0,1}X=\{0\}$, where $T^{0,1}X=\overline{T^{1,0}X}$, and $[\mathcal V,\mathcal V]\subset\mathcal V$, where $\mathcal V$ is the space of smooth sections of the orbibundle $T^{0,1}X$. We also fix a Hermitian metric $\langle\cdot \vert \cdot \rangle$ on $TX\otimes\mathbb{C}$ so that $T^{1,0}X\perp T^{0,1}X$, as before we mean that $\langle\cdot\vert \cdot \rangle$ is defined locally on each local chart $U$ as $G_U$-invariant Hermitian metric on $T_{\tilde{x}}\widetilde{U}\otimes \mathbb{C}$; from now on when it is clear how to define a geometric structure on an orbibundle we will not always recall the definition. Furthermore, notice that every Hermitian metric $h'$ on $\widetilde{U}$ gives rise to a $G_U$-invariant metric $h_U$ on $\widetilde{U}$ by averaging over the group:
	\[h_{{U}}(v,\, w):= \sum_{g\in G_U} h'(\mathrm{d}_{\widetilde{x}}g(v),\,\mathrm{d}_{\widetilde{x}}g(w)),\qquad v,w \in T_{\widetilde{x}}\widetilde{U}\,. \]
	For $u \in TX\otimes \mathbb{C}$, we write $|u|^2 := \langle\, u\, |\, u\, \rangle$.
	
	There is a unique sub-orbibundle $HX$ of $TX$ such that $HX\otimes \mathbb{C}=T^{1,0}X \oplus T^{0,1}X$, i.e. $HX$ is the real part of $T^{1,0}X \oplus T^{0,1}X$. Let $J:HX\To HX$ be the complex structure map given by $J(u+\ol u)=\imath u-\imath\ol u$, for every $u\in T^{1,0}X$. 
	By complex linear extension of $J$ to $TX\otimes \mathbb{C}$, the $\imath$-eigenspace of $J$ is $T^{1,0}X \, = \, \left\{ V \in HX\otimes \mathbb{C} \, ;\, JV \, =  \,  \imath V  \right\}.$ We shall also write $(X, HX, J)$ to denote a CR orbifold.
	
	We  fix a real non-vanishing $1$-form $\omega_0\in\mathcal{C}^{\infty}(X,T^*X)$ so that $\langle\,\omega_0(x)\,,\,u\,\rangle=0$, for every $u\in H_xX$, for every $x\in X$. 
	For each $x \in X$, we define a quadratic form on $HX$ by
	\[
	\mathcal{L}_x(U,V) =\frac{1}{2}\mathrm{d}\omega_0(JU, V), \qquad \forall \ U, V \in H_xX.
	\]
	We extend $\mathcal{L}$ to $HX\otimes \mathbb{C}$ by complex linear extension. Then, for $U, V \in T^{1,0}_xX$,
	\begin{equation}\label{e-gue210813yyd}
		\mathcal{L}_x(U,\overline{V}) = \frac{1}{2}\mathrm{d}\omega_0(JU, \overline{V}) = -\frac{1}{2i}\mathrm{d}\omega_0(U,\overline{V}).
	\end{equation}
	The Hermitian quadratic form $\mathcal{L}_x$ on $T^{1,0}_xX$ is called Levi form at $x$. This means that for each local coordinate patch $(\widetilde{U},G_U)\rightarrow U$ there exists a $G_U$-invariant Hermitian quadratic form on $(T_{\widetilde{x}}\widetilde{U},G_U)$. In this paper, we always assume that the Levi form $\mathcal{L}$ on $T^{1,0}X$ is non-degenerate of constant signature $(n_-,n_+)$ on $X$, where $n_-$ denotes the number of negative eigenvalues of the Levi form and $n_+$ denotes the number of positive eigenvalues of the Levi form. Let $R\in\mathcal{C}^\infty(X,TX)$ be the non-vanishing vector field determined by 
	\begin{equation}\label{e-gue170111ry}\begin{split}
			\omega_0(R)=-1,\quad
			\mathrm{d}\omega_0(R,\cdot)\equiv0\ \ \mbox{on $TX$}.
	\end{split}\end{equation}
	Note that $X$ is a contact orbifold with contact form $\omega_0$, contact plane $HX$ and $R$ is the Reeb vector field. We have $\langle\,R\,|\,R\,\rangle=1$ and $R$ is orthogonal to $T^{1,0}X\oplus T^{0,1}X$. 
	
	Denote by $T^{*1,0}X$ and $T^{*0,1}X$ the dual bundles of $T^{1,0}X$ and $T^{0,1}X$, respectively. They can be identified with sub-orbibundles of the complexified cotangent orbibundle $T^*X\otimes \mathbb{C}$. Define the vector orbibundle of $(0,q)$-forms by $T^{*0,q}X := \wedge^qT^{*0,1}X$. The Hermitian metric $\langle\, \cdot\, |\, \cdot \,\rangle$ on $TX\otimes \mathbb{C}$ induces, by duality, a Hermitian metric on $T^*X\otimes \mathbb{C}$ and also on the orbibundles of $(0,q)$ forms $T^{*0,q}X, q=0, 1, \cdots, n$. We shall also denote all these induced metrics by $\langle \,\cdot \,|\, \cdot \,\rangle$. For $u\in T^{*0,q}X$, we write $\abs{u}^2:=\langle\,u\,|\,u\,\rangle$. Note that we have the point-wise orthogonal decompositions:
	\[
	\renewcommand{\arraystretch}{1.2}
	\begin{array}{c}
		T^*X\otimes \mathbb{C} = T^{*1,0}X \oplus T^{*0,1}X \oplus \left\{ \lambda \omega_0;\, \lambda \in \mathbb{C} \right\}, \\
		TX\otimes \mathbb{C} = T^{1,0}X \oplus T^{0,1}X \oplus \left\{ \lambda R;\, \lambda \in \mathbb{C} \right\}.
	\end{array}
	\]
	
	Let $U$ be an open set of $X$. Let $\Omega^{0,q}(U)$ denote the space of smooth sections of $T^{*0,q}X$ over $D$ and let $\Omega^{0,q}_c(U)$ be the subspace of $\Omega^{0,q}(U)$ whose elements have compact support in $U$; more precisely an element $u$ in $\Omega^{0,q}_c(U)$ has a lift $\tilde{u}\in \Omega^{0,q}_c(\tilde{U})$ which is $G_U$-invariant. The metric on $TX\otimes\mathbb{C}$ induces a metric on $\Lambda^q(T^{*}X\otimes \mathbb{C})$, so we have a well-defined orthogonal projector
	\[\pi^{0,q}\,:\,\Lambda^q(T^{*}X\otimes \mathbb{C})\rightarrow T^{*0,q}X\,.  \]
	We can define the tangential Cauchy-Riemann operator $\overline{\partial}_b$ by composing with the exterior differential
	\[\overline{\partial}_b=\pi^{0,q+1}\circ \mathrm{d}\,:\,\Omega^{0,q}(X)\rightarrow \Omega^{0,q+1}(X)\,,\]
	thus, on a local chart $(\widetilde{U},G_U)\rightarrow U$, $\overline{\partial}_b$ acts on $G_U$-invariant functions on $\widetilde{U}$ and commutes with the action of $G_U$ on $\tilde{U}$.
	Let $\mathrm{dV}_X(x)$ be the volume form on $X$ induced by the Hermitian metric $\langle\,\cdot\,|\,\cdot\,\rangle$. For a local chart $(\widetilde{U},G_U)\rightarrow U$, let $\mathrm{dV}_{\widetilde U}(x)$ be the volume form on $\widetilde U$ induced by the Hermitian metric $\langle\,\cdot\,|\,\cdot\,\rangle$.
	The natural global $L^2$ inner product $(\,\cdot\,|\,\cdot\,)$ on $\Omega^{0,q}(X)$ 
	induced by $\mathrm{dV}_X(x)$ and $\langle\,\cdot\,|\,\cdot\,\rangle$ is given by
	\[
	(\,u\,|\,v\,):=\int_X\langle\,u(x)\,|\,v(x)\,\rangle\, \mathrm{dV}_X(x)\,,\quad u,v\in\Omega^{0,q}(X)\,.
	\]
	We denote by $L^2_{(0,q)}(X)$ the completion of $\Omega^{0,q}(X)$ with respect to $(\,\cdot\,|\,\cdot\,)$. It should be clear that the elements in  $L^2_{(0,q)}(X)$ are $L^2$ sections of the orbibundle $\Lambda^q(T^{*}X\otimes \mathbb{C})$, the integral appearing in the definition of the $L^2$ inner product $(\,\cdot\,|\,\cdot\,)$ is meant to be interpreted as in the previous section, see \eqref{eq:integral}.
	We extend $(\,\cdot\,|\,\cdot\,)$ to $L^2_{(0,q)}(X)$ 
	in the standard way. For $f\in L^2_{(0,q)}(X)$, we denote $\norm{f}^2:=(\,f\,|\,f\,)$.
	We extend
	$\ddbar_{b}$ to $L^2_{(0,r)}(X)$, $r=0,1,\ldots,n$, by
	\[
	\ddbar_{b}:{\rm Dom\,}\ddbar_{b}\subset L^2_{(0,r)}(X)\To L^2_{(0,r+1)}(X)\,,
	\]
	where ${\rm Dom\,}\ddbar_{b}:=\{u\in L^2_{(0,r)}(X);\, \ddbar_{b}u\in L^2_{(0,r+1)}(X)\}$ and, for any $u\in L^2_{(0,r)}(X)$, $\ddbar_{b} u$ is defined in the sense of distributions.
	We also write
	\[
	\ol{\pr}^{*}_{b}:{\rm Dom\,}\ol{\pr}^{*}_{b}\subset L^2_{(0,r+1)}(X)\To L^2_{(0,r)}(X)
	\]
	to denote the Hilbert space adjoint of $\ddbar_{b}$ in the $L^2$ space with respect to $(\,\cdot\,|\,\cdot\, )$.
	Let $\Box^{(q)}_{b}$ denote the (Gaffney extension) of the Kohn Laplacian given by
	\begin{equation}\label{e-suIX}
		\renewcommand{\arraystretch}{1.2}
		\begin{array}{c}
			{\rm Dom\,}\Box^{(q)}_{b}=\Big\{s\in L^2_{(0,q)}(X);\, 
			s\in{\rm Dom\,}\ddbar_{b}\cap{\rm Dom\,}\ol{\pr}^{*}_{b},\,
			\ddbar_{b}s\in{\rm Dom\,}\ol{\pr}^{*}_{b}, \, \ol{\pr}^{*}_{b}s\in{\rm Dom\,}\ddbar_{b}\Big\}\,,\\
			\Box^{(q)}_{b}s=\ddbar_{b}\ol{\pr}^{*}_{b}s+\ol{\pr}^{*}_{b}\ddbar_{b}s
			\:\:\text{for $s\in {\rm Dom\,}\Box^{(q)}_{b}$}\,.
		\end{array}
	\end{equation}
	
	\subsection{Group actions on orbifolds}\label{s-gue220801yyd}
	Consider a compact connected Lie group action $G$. A smooth action of $G$ on an orbifold $X$ is a continuous action $\cdot$ between the underlying topological spaces such that for each $g\in G$ the map $g\cdot\,:\,X\rightarrow X$ is smooth by mean of definition in Section \ref{sec:orbi}: this means that locally $(G_U,\widetilde{U})\rightarrow U$, for each $g\in G$, there exists a local lift $(\widetilde{g\cdot}_U, \overline{g\cdot}_U)$ such that $\widetilde{g\cdot}_U : \widetilde{U}\rightarrow \widetilde{V}$ is smooth, $\overline{g\cdot}_U: G_U\rightarrow G_V$ is a group homomorphism, the diagram
	\begin{align*}
		(\widetilde{U},& G_U) \xrightarrow{\widetilde{g\cdot}} (\widetilde{V},G_V) \\
		\downarrow& \qquad\qquad \downarrow \\
		U& \quad\,\,\,\,\xrightarrow{g\cdot} \quad V
	\end{align*}
	commutes and $\widetilde{g\cdot}_U$ is $\overline{g\cdot}_U$-equivariant:
	\[\widetilde{g\cdot}_U(g\cdot x)= \overline{g\cdot}_U(g)\widetilde{g\cdot}_U(x)\,,\qquad g\in G_U,\,x\in\widehat{U}\,. \]
	
	In this work, when $X$ admits a compact Lie group action $G$, we suppose that  
	
	\begin{ass}\label{a-gue170123I}
		We assume $g^\ast\omega_0=\omega_0$ on $X$ and $g_\ast J=Jg_\ast$ on $HX$, for every $g\in G$, where $g^*$ and $g_*$ denote  the pull-back map and push-forward map of $G$, respectively. 
	\end{ass}
	
	Let $\mathfrak{g}$ denote the Lie algebra of $G$ and $e$ denotes the identity element in $G$. For a moment denote by $f:G\times X \rightarrow X$ the action, so that $f(g,x)=g\cdot x$. Fix $x\in X$, define $$\xi_X(x)=\mathrm{d}_{e}f(\cdot,x)\left(\xi\right)\,.$$
	For any $\xi \in \mathfrak{g}$,  $\xi_X$ is the so called vector field on $X$ induced by $\xi$.
	
	By the motivation explained in the introduction, in Theorem~\ref{t-gue220731yyd}, we assume that $X$ admits a CR and transversal $S^1$-action which is locally free, $T\in \mathcal{C}^{\infty}(X,\,TX)$ denotes the global real vector field given by the infinitesimal circle action. In this case, we will take $T$ to be our Reeb vector field $R$. Note that in Theorem~\ref{thm:main}, we do not assume that $X$ admits a group action. 
	
	\subsection{Operators on orbifolds}
	\label{sec: opor}
	
	We first introduce some notations. 
	Let $U\subset X$ be an open set and let $E$ be a vector orbibundle over $U$. As smooth case, let $\mathcal{D}'(U,E)$ denote the space of distribution sections over $U$ and let $\mathcal{E}'(U,E)$ denote the space of distribution sections over $U$ whose elements have compact support in $U$. Let $\mathcal{C}^\infty(U,E)$ denote the space of smooth sections over $U$ and let $\mathcal{C}^\infty_c(U,E)$ denote the space of smooth sections over $U$ whose elements have compact support in $U$. For $s\in\mathbb R$, let $H^s_{{\rm comp\,}}(U,E)$ denote the Sobolev space of sections over $U$ of order $s$ whose elements have compact support in $U$. Put 
	\[H^s_{{\rm loc\,}}(U,E):=\set{u\in\mathcal{D}'(U,E);\,\chi u\in H^s_{{\rm comp\,}}(U,E), \forall\chi\in\mathcal{C}^\infty_c(U)}.\]
	Let $H^s(X,E)$ denote the Sobolev space of sections over $X$ of order $s$.
	Let $F$ be anther vector orbibundle over $U$.
	We write $E\boxtimes F^*$ to denote the vector orbibundle over $U\times U$ with fiber over $U\times U$ consisting of the linear maps from $F_y$ to $E_x$. 
	
	Let $H:\mathcal{C}^\infty_c(U,E)\To\mathcal{D}'(U,E)$ be a continuous operator and let $H(x,y)\in \mathcal{D}'(U\times U,E\boxtimes E^*)$ be the distribution kernel of $H$. 
	We say that $H$ is smoothing on $U$ if $H(x,y)\in \mathcal{C}^\infty(U\times U,E\boxtimes E^*)$. When $H$ is smoothing on $U$, we write $H\equiv0$ (on $U$) or $H(x,y)\equiv0$ (on $U$). 
	More precisely, let $(\widetilde U,G_U)\To U$ be an orbifold chart. $H\equiv0$ means that we can find a $\widetilde H(x,y)\in\mathcal{C}^\infty(\widetilde U\times\widetilde U,\widetilde E\boxtimes\widetilde E^*)$ such that 
	\[H(x,y)=\sum_{h,g\in G_U}\frac{1}{\abs{G_U}}
	\widetilde H(h\cdot\widetilde x, g\cdot\widetilde y),\]
	for all $(x,y)\in U\times U$, where $\pi(\tilde x)=x$, $\pi(\tilde y)=y$, $\pi:\tilde U\To U$ is the natural projection and $\widetilde E$ is the lifting of $E$. Let $G(x,y)\in\mathcal{D}'(U\times U,E\boxtimes E^*)$. We will write $G$ to denote the continuous operator $G: \mathcal{C}^\infty_c(U,E)\To\mathcal{D}'(U,E)$ with distribution kernel $G(x,y)$.
	
	Let $H:\mathcal{C}^\infty(X,T^{*0,q}X)\To\mathcal{D}'(X,T^{*0,q}X)$ be a continuous operator. We say tat $H$ is a complex Fourier integral operator if let $(\widetilde U,G_U)\To U$ be any orbifold chart of $X$, there is a complex Fourier integral operator 
	$\widetilde H_U: \Omega^{0,q}(\tilde U)\To\Omega^{0,q}(\tilde U)$ has a distribution kernel $\widetilde H_U(\widetilde x,\widetilde y)\in\mathcal{D}'(\widetilde U\times\widetilde U,T^{*0,q}X\boxtimes(T^{*0,q}X)^*)$ 
	such that 
	\[ {H}(x,y)\equiv\frac{1}{\abs{G_U}}\sum_{g, h\in G_U}  \widetilde{H}_{{U}}(h\cdot\widetilde{x}, g\cdot\widetilde{y}) \]
	for $(x,y)\in U\times U$, where $\pi(\tilde x)=x$, $\pi(\tilde y)=y$, $\pi:\tilde U\To U$ is the natural projection. 
	
	
	We recall H\"ormander symbol space. Since we work on a given $\widetilde{U}\subset \mathbb{R}^{2n+1}$, we use notation $x$ for an element $x \in \widetilde{U}$ instead of $\tilde{x}$. Let $\widetilde{U}\subset \mathbb{R}^{2n+1}$ be a local coordinate patch with local coordinates $x=(x_1,\ldots,x_{2n+1})$. We recall the following definition.
	
	\begin{defn}\label{d-gue140221a}
		For $m\in\Real$, $S^m_{1,0}(\widetilde{U}\times \widetilde{U}\times\mathbb{R}_+,T^{*0,q}X\boxtimes(T^{*0,q}X)^*)$ 
		is the space of all  smooth functions $a(x,y,t)\in\mathcal{C}^\infty(\widetilde{U}\times \widetilde{U}\times\mathbb{R}_+,T^{*0,q}X\boxtimes(T^{*0,q}X)^*)$ 
		such that, for all compact $K\Subset \widetilde{U}\times \widetilde{U}$ and all $\alpha, \beta\in\mathbb N^{2n+1}_0$, $\gamma\in\mathbb N_0$, 
		there is a constant $C_{\alpha,\beta,\gamma}>0$ such that 
		\[\abs{\pr^\alpha_x\pr^\beta_y\pr^\gamma_t a(x,y,t)}\leq C_{\alpha,\beta,\gamma}(1+\abs{t})^{m-\gamma},\ \ 
		\mbox{for every $(x,y,t)\in K\times\Real_+, t\geq1$}.\]
		Put 
		\[
		S^{-\infty}(\widetilde{U}\times \widetilde{U}\times\mathbb{R}_+,T^{*0,q}X\boxtimes(T^{*0,q}X)^*) :=\bigcap_{m\in\Real}S^m_{1,0}(\widetilde{U}\times \widetilde{U}\times\mathbb{R}_+,T^{*0,q}X\boxtimes(T^{*0,q}X)^*).
		\]
		Let $a_j\in S^{m_j}_{1,0}(\widetilde{U}\times \widetilde{U}\times\mathbb{R}_+,T^{*0,q}X\boxtimes(T^{*0,q}X)^*)$, 
		$j=0,1,2,\ldots$ with $m_j\To-\infty$, as $j\To\infty$. 
		Then there exists $a\in S^{m_0}_{1,0}(\widetilde{U}\times \widetilde{U}\times\mathbb{R}_+,T^{*0,q}X\boxtimes(T^{*0,q}X)^*)$ 
		unique modulo $S^{-\infty}$, such that 
		$a-\sum^{k-1}_{j=0}a_j\in S^{m_k}_{1,0}(\widetilde{U}\times \widetilde{U}\times\mathbb{R}_+,T^{*0,q}X\boxtimes(T^{*0,q}X)^*\big)$ 
		for $k=0,1,2,\ldots$. 
		
		If $a$ and $a_j$ have the properties listed above, then we write $a\sim\sum^{\infty}_{j=0}a_j$ (in the space
		$S^{m_0}_{1,0}\big(\tilde{U}\times \tilde{U}\times\mathbb{R}_+,T^{*0,q}X\boxtimes(T^{*0,q}X)^*\big)$ ). 
		Furthermore, we write
		\[
		s(x, y, t)\in S^{m}_{{\rm cl\,}}(\widetilde{U}\times \widetilde{U}\times\mathbb{R}_+,T^{*0,q}X\boxtimes(T^{*0,q}X)^*)
		\]
		if $s(x, y, t)\in S^{m}_{1,0}(\widetilde{U}\times \widetilde{U}\times\mathbb{R}_+,T^{*0,q}X\boxtimes(T^{*0,q}X)^*)$ and 
		\[
		\renewcommand{\arraystretch}{1.2}
		\begin{array}{lc}
			&s(x, y, t)\sim\sum^\infty_{j=0}s^j(x, y)t^{m-j}\text{ in }S^{m}_{1, 0}
			(\widetilde{U}\times \widetilde{U}\times\mathbb{R}_+\,,T^{*0,q}X\boxtimes(T^{*0,q}X)^*)\,,\\
			&s^j(x, y)\in\mathcal{C}^\infty(\widetilde{U}\times \widetilde{U},T^{*0,q}X\boxtimes(T^{*0,q}X)^*),\ j\in\N_0.\end{array}
		\]
		We simply write $S^{m}_{1,0}$ to denote $S^m_{1,0}(\widetilde{U}\times \widetilde{U}\times\mathbb{R}_+,T^{*0,q}X\boxtimes(T^{*0,q}X)^*)$, $m\in\mathbb R\cup\set{-\infty}$.  
	\end{defn} 
	
	Let $\widetilde E$ be a vector bundle over $\widetilde U$. Let $m\in\Real$, $0\leq\rho,\delta\leq1$. Let $S^m_{\rho,\delta}(T^*\widetilde U,\widetilde E)$
	denote the H\"{o}rmander symbol space on $T^*\widetilde U$ with values in $\widetilde E$ of order $m$ type $(\rho,\delta)$
	and let $S^m_{{\rm cl\,}}(T^*\widetilde U,\widetilde E)$
	denote the space of classical symbols on $T^*\widetilde U$ with values in $\widetilde E$ of order $m$. Let $L^m_{\rho,\delta}(\widetilde U,\widetilde E\boxtimes\widetilde E^*)$ ($L^m_{{\rm cl\,}}(\widetilde U,\widetilde E\boxtimes\widetilde E^*)$) be the space of pseudodifferential operators on $\widetilde U$ of order $m$ from sections of $E$ to sections of $E$ with full symbol $a\in S^m_{\rho,\delta}(T^*\widetilde U,\widetilde E)$ ($a\in S^m_{{\rm cl\,}}(T^*\widetilde U,\widetilde E)$). Let $E$ be a vector orbibundle over $X$. Let $A: \mathcal{C}^\infty(X,E)\To\mathcal{D}'(X,E)$ be a continuous operator with distribution 
	kernel $A(x,y)\in\mathcal{D}'(X\times X,E\boxtimes E^*)$. We write $A\in L^m_{\rho,\delta}(X,E\otimes E^*)$ ($A\in L^m_{{\rm cl\,}}(X,E\otimes E^*)$) if for any orbifold chart $(\widetilde U,G_U)\To U$, there is a $\widetilde A_U\in L^m_{\rho,\delta}(\widetilde U,\widetilde E\boxtimes\widetilde E^*)$ ($\widetilde A_U\in L^m_{{\rm cl\,}}(\widetilde U,\widetilde E\boxtimes\widetilde E^*)$) such that 
	\[ A(x,y)\equiv\frac{1}{\abs{G_U}}\sum_{g, h\in G_U}  \widetilde{A}_{{U}}(h\cdot\widetilde{x}, g\cdot\widetilde{y}) \]
	for $(x,y)\in U\times U$, where $\pi(\tilde x)=x$, $\pi(\tilde y)=y$, $\pi:\tilde U\To U$ is the natural projection and $\widetilde E$ is the lifting of $E$.

	\subsection{Microlocal Szeg\H{o} kernel}
	
	On each open set $U\subseteq X$, we define operators $S^{(q)}_U$ having the same microlocal structure described in our main results. We call this operators \textit{approximate Szeg\H{o} kernels}. We shall review some results in~\cite{hsiao} concerning the existence of a microlocal Hodge decomposition of the Kohn Laplacian on an open set of a CR manifold where the Levi form is non-degenerate. Since the proofs are local we can carry out the same analysis on the chart $(G_U, \widetilde{U})$ for a given open set $U$. Since we work on a fix $\widetilde{U}\subset \mathbb{R}^{2n+1}$, we use notation $x$ for an element $x \in \widetilde{U}$. For the following theorems we refer to Chapter 6, Chapter 7 and Chapter 8 of Part I in \cite{hsiao}.
	
	\begin{thm} \label{thm:BoxAequalsI}
		We assume that the $G_U$-invariant Levi form is non-degenerate of constant signature $(n_-,n_+)$ at each point of $\widetilde{U}$. Let $q\neq n_-, n_+$. Then, there is a properly supported operator 
		$A_{\widetilde U}\in L^{-1}_{\frac{1}{2},\frac{1}{2}}(\widetilde{U},T^{*0,q}X\boxtimes(T^{*0,q}X)^*)$ such that $\Box^{(q)}_bA_{\widetilde U}=I+F_{\widetilde U}$ on $\widetilde{U}$,
		$F_{\widetilde U}$ is a smoothing operator on $\widetilde U$.
	\end{thm}
	
	Until further notice, we assume that the Levi form is non-degenerate of constant signature $(n_-,n_+)$ on $X$. 
	Let $p_0(x, \xi)\in C^\infty(T^*X)$ be the principal symbol of $\Box^{(q)}_b$. Note that $p_0(x,\xi)$ is a $G_U$-invariant polynomial of degree $2$ in $\xi$ when written locally on $\widetilde{U}$. Recall that the characteristic manifold of $\Box^{(q)}_b$ is given by $\Sigma=\Sigma^+\cup\Sigma^-$, where $\Sigma^+$ and $\Sigma^-$ are given by 
	\begin{equation}\label{e-gue140205I}
		\Sigma^+=\set{(x, \lambda\,\omega_0(x))\in T^*X;\, \lambda>0},\:\:
		\Sigma^-=\set{(x, \lambda\,\omega_0(x))\in T^*X;\, \lambda<0},
	\end{equation}
	where $\omega_0\in\mathcal{C}^\infty(X,T^*X)$ is the uniquely determined global contact $1$-form.
	
	\begin{thm} \label{t-gue140205I}
		With the same assumptions as in the previous theorem, let $q=n_-$ or $n_+$. Then, by adopting notation as in Section \ref{sec: opor}, there exist properly supported continuous operators 
		$A\in  L^{-1}_{\frac{1}{2},\frac{1}{2}}(\widetilde{U},T^{*0,q}X\boxtimes(T^{*0,q}X)^*)$, and operators
		$S_-, S_+\in L^{0}_{\frac{1}{2},\frac{1}{2}}(\tilde{U},T^{*0,q}X\boxtimes(T^{*0,q}X)^*)$, such that
		\begin{equation}\label{e-gue140205II}
			\begin{split}
				&\mbox{$\Box^{(q)}_bA+S_-+S_+=I$ on $\widetilde{U}$},\\
				&\Box^{(q)}_bS_-\equiv0\ \ \mbox{on $\widetilde{U}$},\ \ \Box^{(q)}_bS_+\equiv0\ \ \mbox{on $\tilde{U}$},\\
				&A\equiv A^*\ \ \mbox{on $\widetilde U$},\ \ S_-A\equiv0\ \ \mbox{on $\widetilde U$}, \ \ S_+A\equiv0\ \ \mbox{on $\widetilde{U}$},\\
				&S_-\equiv S_-^*\equiv S_-^2\ \ \mbox{on $\widetilde{U}$},\\
				&S_+\equiv S_+^*\equiv S_+^2\ \ \mbox{on $\widetilde{U}$},\\
				&S_-S_+\equiv S_+S_-\equiv0\ \ \mbox{on $\widetilde{U}$},
			\end{split}
		\end{equation}
		where $A^*$, $S_-^*$ and $S_+^*$ are the formal adjoints of $A$, $S_-$ and $S_+$
		with respect to $(\,\cdot\,|\,\cdot\,)$ respectively and $S_-(x,y)$ satisfies
		\[S_-(x, y)\equiv\int^{\infty}_{0}e^{i\varphi_-(x, y)t}s_-(x, y, t)dt\ \ \mbox{on $\widetilde{U}$}\]
		with a symbol $s_-(x, y, t)\in S^{n}_{{\rm cl\,}}\big(\widetilde{U}\times \widetilde{U}\times\mathbb{R}_+,T^{*0,q}X\boxtimes(T^{*0,q}X)^*\big)$  such that
		\begin{equation}  \label{e-gue140205III}\begin{split}
				&s_-(x, y, t)\sim\sum^\infty_{j=0}s^j_-(x, y)t^{n-j}\quad\text{ in }S^{n}_{1, 0}
				\big(\widetilde{U}\times \widetilde{U}\times\mathbb{R}_+\,,T^{*0,q}X\boxtimes (T^{*0,q}X)^*\big)\,,\\
				&s^j_-(x, y)\in C^\infty\big(\widetilde{U}\times \widetilde{U},T^{*0,q}X\boxtimes (T^{*0,q}X)^*\big),\ \ j\in\N_0,
		\end{split}\end{equation}
		and phase function $\varphi_-$ such that $\varphi=\varphi_-$ satisfies 
		\begin{equation}\label{e-gue140205IV}
			\begin{split}
				&\varphi\in\mathcal{C}^\infty(\widetilde{U}\times \widetilde{U}),\ \ {\rm Im\,}\varphi(x, y)\geq0,\\
				&\varphi(x, x)=0,\ \ \varphi(x, y)\neq0\ \ \mbox{if}\ \ x\neq y,\\
				&d_x\varphi(x, y)\big|_{x=y}=-\omega_0(x), \ \ d_y\varphi(x, y)\big|_{x=y}=\omega_0(x), \\
				&\varphi(x, y)=-\ol\varphi(y, x).
			\end{split}
		\end{equation}
		Moreover, there is a function $f\in\mathcal{C}^\infty(\widetilde{U}\times \widetilde{U})$ such that
		\begin{equation} \label{e-gue140205V}
			p_0(x, \varphi'_x(x,y))-f(x,y)\varphi(x,y)
		\end{equation}
		vanishes to infinite order at $x=y$.
		Similarly,
		\[S_+(x, y)\equiv\int^{\infty}_{0}\!\! e^{i\varphi_+(x, y)t}s_+(x, y, t)dt\ \ \mbox{on $\widetilde{U}$}\]
		with $s_+(x, y, t)\in S^{n}_{{\rm cl\,}}\big(\widetilde{U}\times \widetilde{U}\times\mathbb{R}_+,T^{*0,q}X\boxtimes(T^{*0,q}X)^*\big)$ satisfying 
		\begin{equation}\label{e-gue140309}
			s_+(x, y, t)\sim\sum^\infty_{j=0}s^j_+(x, y)t^{n-j}
		\end{equation}
		in $S^{n}_{1, 0}\big(\widetilde{U}\times \widetilde{U}\times\mathbb{R}_+,T^{*0,q}X\boxtimes (T^{*0,q}X)^*\big)$,
		\[s^j_+(x, y)\in\mathcal{C}^\infty\big(\widetilde{U}\times \widetilde{U},T^{*0,q}X\boxtimes (T^{*0,q}X)^*\big),\ \ j\in\N_0,\]
		and $-\ol\varphi_+(x, y)$ satisfies \eqref{e-gue140205IV} and \eqref{e-gue140205V}. 
		Moreover, if $q\neq n_+$, then $s_+(x,y,t)$ vanishes to infinite order at $x=y$. 
		If $q\neq n_-$, then $s_-(x,y,t)$ vanishes to infinite order at $x=y$.
	\end{thm}
	
	The operators $S_{+}$, $S_{-}$ are called \emph{approximate Szeg\H{o} kernels}.
	
	\begin{rem}\label{r-gue140211}
		With the notations and assumptions used in Theorem~\ref{t-gue140205I}, furthermore suppose that $q=n_-\neq n_+$. Since
		$s_+(x,y,t)$ vanishes to infinite order at $x=y$, we have $S_+\equiv0$ on $\widetilde{U}$. 
		Similarly, if $q=n_+\neq n_-$, then $S_-\equiv0$ on $\widetilde{U}$.
	\end{rem}
	
	We pause and introduce some notations. For a given point $p\in\widetilde{U}$, let $\{W_j\}_{j=1}^{n}$ be an
	orthonormal frame of $(T^{1,0}\widetilde{U},\langle\,\cdot\,|\,\cdot\,\rangle)$ near $p$, for which the Levi form
	is diagonal at $p$. Put
	\[
	\mathcal{L}_{p}(W_j,\ol W_\ell)=\mu_j(p)\delta_{j,\ell}\,,\;\; j,\ell=1,\ldots,n\,.
	\]
	We will denote by
	\begin{equation}\label{det140530}
		\det\mathcal{L}_{p}=\prod_{j=1}^{n}\mu_j(p)\,.
	\end{equation}
	Let $\{e_j\}_{j=1}^{n}$ denote the basis of $T^{*0,1}\widetilde{U}$, dual to $\{\ol W_j\}^{n}_{j=1}$. We assume that
	$\mu_j(p)<0$ if $1\leq j\leq n_-$ and $\mu_j(p)>0$ if $n_-+1\leq j\leq n$. Put
	\[
	\renewcommand{\arraystretch}{1.2}
	\begin{array}{ll}
		&\mathcal{N}(p,n_-):=\set{ce_1(p)\wedge\ldots\wedge e_{n_-}(p);\, c\in\Complex},\\
		&\mathcal{N}(p,n_+):=\set{ce_{n_-+1}(p)\wedge\ldots\wedge e_{n}(p);\, c\in\Complex}
	\end{array}
	\]
	and let
	\begin{equation}\label{tau140530}
		\tau_{p,n_-}:T^{*0,q}_{p}\widetilde{U}\To\mathcal{N}(p,n_-)\,,\quad
		\tau_{p,n_+}:T^{*0,n-q}_{p}\widetilde{U}\To\mathcal{N}(p,n_+)\,,
	\end{equation} is
	be the orthogonal projections onto $\mathcal{N}(p,n_-)$ and $\mathcal{N}(p,n_+)$
	with respect to $\langle\,\cdot\,|\,\cdot\,\rangle$, respectively. For $J=(j_1,\ldots,j_q)$, $1\leq j_1<\cdots<j_q\leq n$, let 
	$e_J:=e_{j_1}\wedge\cdots\wedge e_{j_q}$. For $\abs{I}=\abs{J}=q$, $I$, $J$ are strictly increasing, let $e_I\otimes(e_J)^*$ be the linear transformation from 
	$T^{*0,q}X$ to $T^{*0,q}X$ given by 
	\[(e_I\otimes(e_J)^*)(e_K)=\delta_{J,K}e_I,\]
	for every $\abs{K}=q$, $K$ is strictly increasing, where $\delta_{J,K}=1$ if $J=K$, $\delta_{J,K}=0$ if $J\neq K$. For any $f\in T^{*0,q}X\boxtimes(T^{*0,q}X)^*$, we have \[f=\sideset{}{'}\sum_{\abs{I}=\abs{J}=q}c_{I,J}e_I\otimes(e_J)^*,\]
	$c_{I,J}\in\mathbb C$, for all $\abs{I}=\abs{J}=q$, $I$, $J$ are strictly increasing, where $\sum'$ means that the summation is performed only over strictly increasing multi-indices.  We call $c_{I,J}e_I\otimes(e_J)^*$ the component of $f$ in the direction $e_I\otimes(e_J)^*$. Let $I_0=(1,2,\ldots,q)$. We can check that 
	\[\tau_{p,n_-}=e_{I_0}(p)\otimes(e_{I_0}(p))^*.\]

	The following formula for the leading term $s^0_-$ on the diagonal follows from \cite[\S 9]{hsiao}. The formula for the leading term $s^0_+$ on the diagonal follows similarly.
	
	\begin{thm} \label{t-gue140205III}
		Let $q=n_-$. 
		For the leading term $s^0_{-}(x,y)$ of the expansion \eqref{e-gue140205III} of $s_{-}(x,y,t)$, we have
		\[
		s^0_{-}(x_0, x_0)=\frac{1}{2}\pi^{-n-1}\abs{\det\mathcal{L}_{x_0}}\tau_{x_0,n_{-}}\,,\:\:x_0\in \widetilde U.
		\] 
		
		Let $q=n_+$. For the leading term $s^0_{+}(x,y)$ of the expansion \eqref{e-gue140309} of $s_{+}(x,y,t)$, we have
		\[
		s^0_{+}(x_0, x_0)=\frac{1}{2}\pi^{-n-1}\abs{\det\mathcal{L}_{x_0}}\tau_{x_0,n_{-}}\,,\:\:x_0\in \widetilde U.
		\] 
	\end{thm} 
	
	\subsection{Geometric conditions for the closed range property}
	
	We recall that, given $q\in\{0,\ldots,n\}$, the Levi form is said to satisfy condition $Y(q)$ at $p\in X$, if $\mathcal{L}_{p}$ has at least either $\min{(q+1,n-q)}$ pairs of eigenvalues with opposite signs or $\max{(q+1, n-q)}$ eigenvalues of the same sign.  
	
	\begin{rem}
		Assume that the Levi form is non degenerate of constant signature $(n_-,n_+)$. Given $q\in \{0,\,\dots,\, n\}$, then $Y(q)$ holds if $q\notin \{n_-,\,n_+\}$. Furthermore, notice that if $q=n_-=n_+$, $q-1\notin \{n_-,n_+\}$ and $q+1\notin \{n_-,n_+\}$, then $Y(q-1)$ and $Y(q+1)$ hold. Eventually, we remark that if $q=n_-$ and $\lvert n_--n_+\rvert >1$, then then $Y(q-1)$ and $Y(q+1)$ hold. In particular if $X$ is a strongly pseudoconvex CR orbifold of dimension greater than $5$, then $Y(1)$ holds.
	\end{rem} 
	
	We can repeat the same proof as in manifold case with small modification and get the following:
	
	\begin{thm} \label{thm:closedrange}
		Let $X$ be a compact CR orbifold and suppose that $Y(q)$ holds. Then, $\Box^{(q)}_b$ has closed range. 
		
		Furthermore, suppose that $Y(q)$ fails but $Y(q-1)$ and $Y(q+1)$ hold, then $\Box^{(q)}_b$ has closed range.
	\end{thm}

	In the following, we will give a proof of the first pat of Theorem~\ref{thm:closedrange} under the assumption that 
	the Levi form is non-degenerate of constant signature $(n_-,n_+)$. Since the proof is related to the construction of  parametrices for Kohn Lapalcians on orbifolds, therefore we think it is worth to give a proof under the assumption that the Levi form is non-degenerate of constant signature $(n_-,n_+)$.
	We suppose that $Y(q)$ holds. Then, $q\notin\set{n_-,n_+}$. We shall prove:
	\begin{equation} \label{eq:closedrange}
		\exists \,C>0 \quad\text{such that}\quad \lVert\square^{(q)}_b u \rVert\geq C\, \lVert u\rVert \quad \forall u\perp \Ker \square^{(q)}_b\,;
	\end{equation}
	which is equivalent to prove that $\square^{(q)}_b$ has closed range, see Lemma C.1.1 in \cite{mm}. We need the following lemma.
	
	\begin{lem} \label{lem:A}
		There exists a continuous operator 
		\[A: H^s(X,T^{*0,q}X)\To H^{s+1}(X,T^{*0,q}X),\ \ \forall s\in\mathbb R,\]
		such that $\square^{(q)}_bA=I+F$ where $F$ is smoothing on $X$. 
	\end{lem} 
	
	\begin{proof}[Proof of Lemma \ref{lem:A}]
		Fix an orbifold chart $(G_U,\tilde{U})\rightarrow U$.
		By Theorem \ref{thm:BoxAequalsI}, we can define on $\widetilde{U}$,
		\[\hat{A}_U(x,y)=\sum_{g,h\in G_U} \frac{1}{\lvert G_U\rvert}A_{\widetilde U}
		(h\cdot x,g\cdot y)\quad\text{ and }\quad \hat{F}_U(x,y)=\sum_{g,h\in G_U} \frac{1}{\lvert G_U\rvert}F_{\widetilde U}(h\cdot x,g\cdot y),\]
		so that $\hat{A}_U:H^s_{{\rm comp\,}}(\widetilde{U},T^{*0,q}X)\rightarrow 
		H^{s+1}_{{\rm loc\,}}(\widetilde{U},T^{*0,q}X)$ is continuous, for all $s\in\mathbb R$, $\hat{F}_U$ is smoothing and they satisfy the following
		\[ \Box^{(q)}_b\hat{A}_U= I+\hat{F}_U\quad \text{ on }\quad U\,,\]
		where $\hat{A}_U$ and $\hat F_U$ are as in Theorem~\ref{thm:BoxAequalsI}. 
		Now, since $X$ is compact we can consider an open cover $\{U_j\}_{j\in J}$, $(\widetilde U_j,G_{U_j})\To U_j$ orbifold chart, for every $j$, and a partition of unity $\chi_j\in \mathcal{C}_c^{\infty}(U_j)$ so that we have
		\[\Box^{(q)}_b\hat{A}_{U_j}\chi_j=\chi_j+\hat{F}_{U_j}\chi_j\quad \text{ on }\quad U_j\,. \]

		Eventually, set $\hat{A}=\sum_j \hat{A}_j\chi_j$ and $\hat{F}=\sum_j \hat{F}_j\chi_j$, we can check
		\[\Box^{(q)}_b\hat{A}= I+\hat{F}\,,\]
		globally on $X$ and we get the lemma.
	\end{proof}
	
	We are ready to prove \eqref{eq:closedrange}. Suppose by the contrary that \eqref{eq:closedrange} does not hold. Then there exists a sequence of $(0,q)$ forms such that 
	$u_j\in L^2_{(0,q)}(X)$, $j=1,2,\ldots$,  with $\lVert u_j\rVert=1$ and $u_j \perp \Ker \square^{(q)}_b$ such that $\lVert\square^{(q)}_b u_j \rVert< \frac{1}{j}\lVert u_j\rVert $ for each $j=1,2\dots$. By Lemma \ref{lem:A}, 
	we have
	\[\hat{A}^*\square_b^{(q)}u_j=u_j+\hat{F}^*u_j,\ \ j=1,2,\ldots,\]
	where $\hat A^*$ and $\hat F^*$ are adjoints of $\hat A$ and $\hat F$ respectively. 
	Since $\hat{A}^*\,:\,L^2_{(0,q)}(X)\rightarrow H^1(X,T^{*0,1}X)$ is continuous, we get
	\[\lVert u_j\rVert_{1} - \lVert \hat{F}^* u_j \rVert_{1}\leq\lVert \hat{A}^* \square^{(q)}_b u_j \rVert_1\leq C \lVert\square_b^{(q)}u_j\rVert\leq \frac{C}{j}\lVert u_j\rVert \,,\]
	for all $j=1,2,\ldots$, where $\norm{\cdot}_1$ denotes the Sobolev norm of order one 
	and $C>0$ is a constant independent of $u_j$.
	Since $\hat{F}:L^2_{(0,q)}(X)\rightarrow H^1(X,T^{*0,q}X)$ is continuous, we get
	\[ \lVert u_j\rVert_{1} \leq C',\ \ \mbox{for all $j=1,2,\ldots$},\]
	where $C'>0$ is a constant independent of $j$. 
	By Rellich's Lemma, there is a subsequence $1\leq j_1<j_2<\dots$ such that $u_{j_s}\rightarrow u$ in $L^2_{(0,q)}(X)$ as $s\rightarrow +\infty$. Since $u_j\perp{\rm Ker\,}\square^{(q)}_b$ for each $j$, then $u\perp {\rm Ker\,}\square^{(q)}_b$; but $\square^{(q)}_b u=\lim_s \square^{(q)}_b u_{j_s}=0$ and thus we get a contradiction. 
	
	We can repeat the proof of~\cite[Theoredm 6.24]{HM} with minor change and get the following
	
	\begin{thm}\label{t-gue220730yyd}
		Assume that the Levi form is non-degenerate of constant signature $(n_-,n_+)$ on $X$. Suppose that $X$ admits a transversal and CR $\mathbb R$-action. For any $q\in\set{0,1,\ldots,n}$,  \[\Box^{(q)}_b:{\rm Dom\,}\Box^{(q)}_b\subset L^2_{(0,q)}(X)\To L^2_{(0,q)}(X)\] 
		has closed range. 
	\end{thm}

	\section{Proof of Theorem~\ref{thm:main}}\label{s-gue220730yyd}
	
	Let 
	\[\Pi^{(q)}: L^2_{(0,q)}(X)\To{\rm Ker\,}\Box^{(q)}_b\]
	be the orthogonal projection (Szeg\H{o} projection). 
	Let us recall the following known global result.
	\begin{thm} \label{thm:global}
		Suppose that the Kohn Laplacian $\square_b^{(q)}:{\rm Dom\,}\square_b^{(q)}\subset L^2_{(0,q)}(X)\rightarrow L^2_{(0,q)}(X)$ has closed range, then there exits a bounded operator $N:L^2_{(0,q)}(X)\rightarrow{\rm Dom\,}\square_b^{(q)}$ such that
		\[N\square_b^{(q)}+\Pi^{(q)}=I\quad\text{on }{\rm Dom\,}\square_b^{(q)} \quad\text{ and }\quad \square_b^{(q)}N+\Pi^{(q)}=I\quad\text{on }L^2_{(0,q)}(X)  \,. \]
	\end{thm}
	
	Consider an orbifold chart $(G_U,\widetilde{U})\rightarrow U$, let 
	\[S_-, S_+, A: \Omega^{0,q}_c(\widetilde{U})\rightarrow \Omega_c^{0,q}(\widetilde{U})\]
	be as in Theorem \ref{t-gue140205I}. Recall that  $S_-,S_+$ and $A$ are properly supported on $\widetilde{U}$. Define on $\widetilde{U}$ the following $G_U$-invariant smooth kernels
	\[\begin{split}
		&\hat{S}_-(x,y)=\sum_{h,g\in G_U}\frac{1}{\abs{G_U}}S_-(h\cdot x,g\cdot y),\quad \hat{S}_+(x,y)=\sum_{h,g\in G_U}\frac{1}{\abs{G_U}}S_+(h\cdot x,g\cdot y),\\ &\hat{A}(x,y)=\sum_{h,g\in G_U}\frac{1}{\abs{G_U}} A(h\cdot x,g\cdot y)\,. \end{split}\]
	
	It is not difficult to see that $\hat{S}_-$, $\hat S_+$ and $A$ are properly supported. For every $u \in \Omega^{0,q}_c(U)$, we have
	\[\hat{A}u=\sum_{h,g\in G_U}\frac{1}{\abs{G_U}}\int_{\widetilde{U}} A(h\cdot x,g\cdot y)u(y)\mathrm{dV}_{\widetilde U}(y)=\sum_{h,g\in G_U}\frac{1}{\abs{G_U}}\int_{\widetilde{U}} A(h\cdot x, y)u(g^{-1}\cdot y)\mathrm{dV}_{\widetilde U}(y),\]
	which is in $\Omega^{0,q}_c(U)$ because $A$ is proper and $u\circ g^{-1}$ is compactly supported. Similarly one can deal with $\hat{S}_-$ and $\hat{S}_+$. Thus, we have $\hat{S}_-,\hat{S}_+,\hat{A}\,:\,\Omega^{0,q}_c(U)\rightarrow \Omega_c^{0,q}(U)$. 
	
	By the properties of $S_-,S_+$ and $A$, we have
	\begin{equation} \label{eq:1}
		\square_b^{(q)}\hat{A}+\hat{S}_-+\hat{S}_+=I\quad\text{on }\tilde{U}
	\end{equation}
	and 
	\begin{equation} \label{eq:2}
		\square_b^{(q)}\hat{S}_- \equiv 0 \quad\text{on }\tilde{U}\,.
	\end{equation}
	Now, apply $\Pi^{(q)}$ to \eqref{eq:1}, for each $u\in \Omega_c^{0,q}(U)$ we have
	\begin{equation} \label{eq:pias}
		\Pi^{(q)}(\square_b^{(q)}\hat{A}+\hat{S}_-+\hat{S}_+)u=\Pi^{(q)}u\,,
	\end{equation} 
	where $\Pi^{(q)}u\in L^2_{(0,q)}(X)$ and the left hand side is well-defined since we have already checked that $\hat{S}_-,\hat{S}_+,\hat{A}$ are properly supported. Now, notice that for each $u,v \in \Omega_c^{0,q}(U)$ we have $(\,\Pi^{(q)}\square_b^{(q)}\hat{A}u\,|\, v\,)=(\,\hat{A}u\,|\,\square_b^{(q)}\Pi v\,)=0$ which implies $\Pi^{(q)}\square_b^{(q)}\hat{A}=0$ on $\Omega_c^{0,q}(U)$. Thus, by \eqref{eq:pias}, we also get the following desired property
	\begin{equation} \label{eq:3}
		\Pi^{(q)}(\hat{S}_-+\hat{S}_+)=\Pi^{(q)}\quad\text{on }\Omega_c^{0,q)}(U)\,.
	\end{equation}
	
	Recall from Theorem \ref{thm:global} the existence of the operator $N$; since $\hat{S}_-$ and $\hat{S}_+$ are properly supported we have
	\[N\square_b^{(q)}(\hat{S}_-+\hat{S}_+) + \Pi^{(q)}(\hat{S}_-+\hat{S}_+)=\hat{S}_-+\hat{S}_+ \quad \text{ on }\Omega_c^{0,q}(U)\,.  \]
	Now, we get that $\square_b^{(q)}(\hat{S}_-+\hat{S}_+)$ is a smoothing operator on $U$, which we will denote by $F$. Thus, by \eqref{eq:3} we write
	\begin{equation} \label{eq:4}
		NF + \Pi^{(q)}=\hat{S}_-+\hat{S}_+=: \hat{S} \quad \text{ on }\Omega_c^{0,q}(U)\,.
	\end{equation}
	We recall that since $F$ is smoothing and properly supported on $U$, it maps $\mathcal{E}'(U,T^{*0,q}X)$ to compactly supported smooth sections, so $F:\mathcal{E}'(U,T^{*0,q}X) \rightarrow \Omega_c^{0,q}(U)$. Thus, by composition we get in \eqref{eq:4} an operator $NF:\mathcal{E}'(U,T^{*0,q}X)\rightarrow L^2_{(0,q)}(X)$. Taking the adjoint in \eqref{eq:4}, we get $\hat{S}^*-\Pi^{(q)}=F^*N$ on $\Omega_c^{0,q}(U)$ and thus by making use again of \eqref{eq:4} we obtain 
	\begin{equation} \label{eq:fnnf}
		(\hat{S}^*-\Pi^{(q)})(\hat{S}-\Pi^{(q)})=F^*N^2F \quad \text{ on } \Omega_c^{0,q}(U)
	\end{equation}
	where it is easy to see that the right hand side is smoothing. Thus, the left hand side of \eqref{eq:fnnf} is also smoothing and we have
	\[\hat{S}^*\hat{S}-\hat{S}^*\Pi^{(q)}-\Pi^{(q)}\hat{S} +\Pi^2 \equiv 0\,. \]
	Eventually by \eqref{eq:3}  we have 
	\begin{equation}\label{e-gue220730ycd}
		\Pi^{(q)}\equiv \hat{S}^*\hat{S}.
	\end{equation}
	From \eqref{eq:1} and $\hat S^*\Box^{(q)}_b\equiv0$, we get 
	\begin{equation}\label{e-gue220730ycdI}
		\hat{S}^*\hat{S}\equiv\hat S.
	\end{equation}
	From \eqref{e-gue220730ycd} and \eqref{e-gue220730ycdI},we get 
	\begin{equation}\label{e-gue220730ycdII}
		\Pi^{(q)}\equiv  \hat{S}_-+\hat{S}_+.
	\end{equation} 
	
	Define on $\widetilde{U}$ the following kernels
	\[
	\tilde{S}_-(x,y)=\sum_{g\in G_U}S_-(x,g\cdot y),\quad \tilde{S}_+(x,y)=\sum_{g\in G_U}S_+(x,g\cdot y).\]
	Note that $\tilde{S}_-(x,y)$ and $\tilde{S}_+(x,y)$ are not $G_U$-invariant. 
	Put $\tilde{S}:=\tilde{S}_-+\tilde{S}_+$. 
	
	\begin{lem}\label{l-gue220730yyd}
		With the notations used above, we have 
		\begin{equation}\label{e-gue220730ycda}
			\mbox{$\hat S\equiv\tilde S$ on $\tilde U$.}
		\end{equation}
	\end{lem} 
	
	\begin{proof}
		The open set $\widetilde U$ is a non-compact CR manifold. Let 
		\[\Box^{(q)}_{b,\widetilde U}: {\rm Dom\,}\Box^{(q)}_{b,\widetilde U}\subset L^2_{(0,q)}(\widetilde U)\To L^2_{(0,q)}(\widetilde U)\]
		be the Gaffney extension of Kohn Laplacian on $\widetilde U$ with respect to $\langle\,\cdot\,|\,\cdot\,\rangle$ and $\mathrm{dV}_{\widetilde U}$. $\Box^{(q)}_{b,\widetilde U}$ is self-adjoint. Fix $\lambda>0$ and set $\Pi^{(q)}_{\leq\lambda,\widetilde U}:=1_{[0,\lambda]}(\Box^{(q)}_{b,\widetilde U})$, where 
		$1_{[0,\lambda]}(\Box^{(q)}_{b,\widetilde U})$ denotes the functional calculus of $\Box^{(q)}_{b,\widetilde U}$ with respect to $1_{[0,\lambda]}$. From~\cite[Theorem 4.7]{HM}, we have 
		\begin{equation}\label{e-gue220730ycde}
			\Pi^{(q)}_{\leq\lambda,\widetilde U}\equiv S_-+S_+\ \ \mbox{on $\widetilde U$}. 
		\end{equation}
		Let $L^2_{(0,q)}(\widetilde U)^{G_U}$ be the space of $G_U$-invariant $L^2$ $(0,q)$ forms. Let 
		\[Q: L^2_{(0,q)}(\widetilde U)\To L^2_{(0,q)}(\widetilde U)^{G_U}\]
		be the orthogonal projection. We claim that 
		\begin{equation}\label{e-gue220722ycdf}
			\mbox{$Q\circ\Pi^{(q)}_{\leq\lambda,\widetilde U}=\Pi^{(q)}_{\leq\lambda,\widetilde U}\circ Q$ on $L^2_{(0,q)}(\widetilde U)$}. 
		\end{equation} 
		Let $\tau\in\mathcal{C}^\infty_c(]0,\lambda[)$. Since $Q$ commutes with $\Box^{(q)}_{b,\widetilde U}$, we have 
		\begin{equation}\label{e-gue220722ycdg}
			Q(z-\Box^{(q)}_{b,\widetilde U})^{-1}=(z-\Box^{(q)}_{b,\widetilde U})^{-1}Q\ \ \mbox{on $L^2_{(0,q)}(\widetilde U)$},
		\end{equation}
		for every $z\in\mathbb C$ with ${\rm Im\,}z\neq0$. From \eqref{e-gue220722ycdg} and Helffer-Sj\"ostrand formula, we have 
		\begin{equation}\label{e-gue220722ycdh}
			\begin{split}
				&\tau(\Box^{(q)}_{b,\widetilde U})Q=\frac{1}{2\pi i}\int_{\mathbb C}\frac{\pr\tilde\tau}{\pr\ol z}(z-\Box^{(q)}_{b,\widetilde U})^{-1}Q dz\wedge d\overline z\\
				&=\frac{1}{2\pi i}\int_{\mathbb C}Q
				\frac{\pr\tilde\tau}{\pr\ol z}(z-\Box^{(q)}_{b,\widetilde U})^{-1}dz\wedge d\ol z=Q\tau(\Box^{(q)}_{b,\widetilde U}),
			\end{split}
		\end{equation}
		where $\tilde\tau$ is an almost analytic extension of $\tau$.
		From \eqref{e-gue220722ycdh}, we get the claim \eqref{e-gue220722ycdf}. 
		
		From \eqref{e-gue220730ycde} and \eqref{e-gue220722ycdf}, we get 
		\begin{equation}\label{e-gue220722ycdp}
			Q\circ (S_-+S_+)\equiv(S_-+S_+)Q\equiv Q(S_-+S_+)Q.
		\end{equation}
		Note that $\hat S=\abs{G_U}Q(S_-+S_+)Q$, $\tilde S=\abs{G_U}(S_-+S_+)Q$. From this observation and \eqref{e-gue220722ycdp}, we get \eqref{e-gue220730ycda}. 
	\end{proof}
	
	From \eqref{e-gue220730ycdII} and \eqref{e-gue220730ycda}, Theorem~\ref{thm:main} follows. 	
	
	\section{Applications}
	
	\subsection{Proof of Theorem \ref{thm:dim}}
		We have the following formula
		\[\dim{\rm Ker\,}\Box^{(q)}_{b,kp}=\int_X {\rm Tr\,}\Pi^{(q)}_{kp}(x,x) \mathrm{dV}_X(x).  \]
		As explained in the introduction, $X$ can be written as a disjoint union based on the orbit type induced by the circle action and $X_{\ell_0}$ is an open and dense subset of $X$. 
		Moreover, from \eqref{eq:szegos1free}, we see that there is a constant $C>0$ independent of $k$ such that 
		\begin{equation}\label{e-gue220801ycd}
			{\rm Tr\,}\Pi^{(q)}_{kp}(x,x)\leq Ck^n,
		\end{equation}
		for all $x\in X$. From \eqref{e-gue220801ycd}, we can apply Lebesgue's dominated convergence theorem and get 
		\begin{equation}\label{e-gue220801ycdI}
			\lim_{k\To+\infty}k^{-n}\dim{\rm Ker,}\Box^{(q)}_{b,kp}=\int_{X_{\ell_0}}
			\lim_{k\To+\infty}k^{-n}{\rm Tr\,}\Pi^{(q)}_{kp}(x,x)\mathrm{dV}_X(x).\end{equation}
		From \eqref{eq:szegos1free} and \eqref{e-gue220801ycdI}, Theorem~\ref{thm:dim} follows. 
	
	\subsection{Proof of Theorem~\ref{thm:genkodaira}}\label{s-gue220805yyd}
	
	We will identify $H^0(M,L^{kp})$ with ${\rm Ker\,}\Box^{(0)}_{b,kp}$ and since we work locally on an orbifold chart $(\widetilde{U},G_U)\rightarrow U$ we always write $x$ for local coordinates on $\tilde{U}$.	Let us first proof that the differential of the map $\Phi_{kp}$ is injective, if $k\gg1$. Consider an open set $U$ around a point $x_0\in X_\ell$ and an orbifold chart $(\widetilde{U},G_U)\rightarrow U$ such that $x_0$ is a fix point of $G_U$.
	Let $x=(x_1,\ldots,x_{2n+1})$ be local coordinates of $\widetilde U$
	with $x(x_0)=0$ and $T={\pr}/{\pr x_{2n+1}}$. Let $z=(z_1,\ldots,z_n)\in\mathbb C^n$, $z_j=x_{2j-1}+ix_{2j}$, $j=1,\ldots,n$. Let $\chi\in\mathcal{C}^\infty_c(\widetilde U)$ with $\int_X\chi\, \mathrm{dV}_X=1$. For $k\in\mathbb N$, put $$\chi_k(x)=\chi(k^{\frac{1}{2}+\varepsilon}z,k^{1+\varepsilon}x_{2n+1})\quad \text{ and }\quad\hat\chi_k(x):=\sum_{g\in G_U}\chi_k(g\cdot x)\,,$$ where $\varepsilon>0$ is a small constant. 
	Let us define 
	\[u_k:= k^{1+(2n+1)\varepsilon}\Pi^{(0)}_{kp}(\hat\chi_k) \in {\rm Ker\,}\Box^{(0)}_{b,kp}.\]
	From \eqref{eq:szegos1free}, we can check that 
	\begin{equation}\label{e-gue220805ycd}
		\lim_{k\To+\infty}u_k(x_0)=p^n\frac{\ell}{2}\pi^{-n-1}\abs{\det\mathcal{L}_{x_0}}.
	\end{equation}
	For each $j=1,\dots n$, it is easy to see that the function
	\[\tilde{f}_j:\widetilde{U}\rightarrow \mathbb{C}\,,\qquad z \mapsto f_j(z)=\sum_{g\in G_U} g\cdot z_j\]
	is $G_U$-invariant and smooth and thus descends naturally to a continuous function $f_j:U \rightarrow \mathbb{C}$ on $U$. Let us also set
	\[u_k^{(j)}=k^{1+(2n+1)\varepsilon}\Pi^{(0)}_{kp}(f_j\hat\chi_k)\in {\rm Ker\,}\Box^{(0)}_{b,k}\,. \]
	From $\Psi(x,x)=0$ and by using integration by parts, we get 
	\begin{equation}\label{eq:uj}
		\begin{split}
			&\lim_{k\To+\infty}(\frac{\partial}{\partial z_q}u^{(j)}_k)(x_0)=\delta_{j,q}p^n\frac{\ell}{2}\pi^{-n-1}\abs{\det\mathcal{L}_{x_0}},\\
			&\lim_{k\To+\infty}u^{(j)}_k(x_0)=0,
		\end{split}
	\end{equation}
	for every $j, q=1,\ldots,n$. 
	Now, consider sections $g_1,\dots,g_{d_{pk}-n-1}$ such that 
	\[\{u_k,u_k^{(1)},\dots,u_k^{(n)}, g_1,\dots,g_{d_{kp}-n-1}\}\]
	is a basis of $H^0(M,L^{kp})$.
	As a consequence of \eqref{e-gue220805ycd} and \eqref{eq:uj}, it is easy to see that the map
	\begin{equation}
		\label{eq:map}
		x \mapsto \left(\frac{u^1_k}{u_k}(x),\dots,\frac{u^n_k}{u_k}(x),\frac{g_1}{u_k}(x),\dots,\frac{g_{d_k-n-1}}{u_k}(x)\right) 
	\end{equation}
	is injective if $k$ is sufficiently large. The Kodaira embedding map is constructed by using an orthonormal basis of $H^0(M,L^{kp})$; the basis used to define the map \eqref{eq:map} is not orthonormal in general. By considering the change of basis matrix we get that $\mathrm{d}\Phi_{kp}$ is injective if $k$ is sufficiently large.
	
	We now need to prove that $\Phi_{kp}$ is globally injective. By absurd, up to passing to a subsequence, suppose that there are $x_k,\,y_k\in M$ with $x_k\neq y_k$ such that $\Phi_{kp}(x_k)=\Phi_{kp}(y_k)$, for each $k$. We prove the theorem by showing that we get a contradiction.
	
	Assume that $\lim_{k\To+\infty}x_k=x_0$, $\lim_{k\To+\infty}y_k=y_0$. If $x_0\neq y_0$, we can repeat the procedure in~\cite{hsiao1} and find $u_k, v_k\in H^0(M,L^{kp})$ such that $\abs{u_k(x_k)}^2_{h^{L^{kp}}}\geq Ck^n$, $\abs{v_k(x_k)}^2_{h^{L^{kp}}}\leq\frac{C}{2}k^n$, $\abs{u_k(y_k)}^2_{h^{L^{kp}}}\leq\frac{C}{2}k^n$, $\abs{v_k(y_k)}^2_{h^{L^{kp}}}\geq Ck^n$, for all $k\gg1$, where $C>0$ is a constant independent of $k$. Thus, $\Phi_{kp}(x_k)\neq\Phi_{kp}(y_k)$, for $k$ large. We get a contradiction and thus we must have $x_0=y_0$. 
	
	For every $k\in\mathbb N$, put 
	\[A_k:=\set{(x,y)\in M\times M;\, \Phi_k(x)=\Phi_k(y)}.\]
	We need 
	
	\begin{lem}\label{l-gue220806yyd}
		There is a constant $k_0$ such that for every $j\in\mathbb N$ and every $k\in\mathbb N$ with $k\geq k_0$, we have
		\[A_{j+k}\subset A_j.\]
	\end{lem} 
	
	\begin{proof}
		From the proof of \eqref{e-gue220805ycd}, we see that there is a $k_0\in\mathbb N$ such that for every $x\in X$ and every $k\in\mathbb N$, $k\geq k_0$, we can find $u_k\in{\rm Ker\,}\Box^{(0)}_{b,kp}$ such that $\abs{u_k(x)}\geq\frac{1}{2}$. Fix $j\in\mathbb N$, let $k\in\mathbb N$ with $k\geq k_0$ and let $(x,y)\notin A_j$. We claim that $(x,y)\notin A_{j+k}$. Since $(x,y)\notin A_j$, we can find 
		$v_j\in{\rm Ker\,}\Box^{(0)}_{b,jp}$ so that $v_j(x)=1$ and $v_j(y)=0$. Let 
		$u_k\in{\rm Ker\,}\Box^{(0)}_{b,kp}$ such that $\abs{u_k(x)}\geq\frac{1}{2}$. Then, $u_kv_j\in{\rm Ker\,}\Box^{(0)}_{b,(j+k)p}$, 
		$\abs{(u_kv_j)(x)}\geq\frac{1}{2}$, $\abs{(u_kv_j)(y)}=0$. Thus, $(x,y)\notin A_{j+k}$. The lemma follows. 
	\end{proof} 
	
	We now have $\lim_{k\To+\infty}x_k=\lim_{k\To+\infty}y_k=0$. 
	Consider an open set $U$ around a given point $x_0\in X_\ell$ and an orbifold chart $(\widetilde{U},G_U)\rightarrow U$ such that $x_0$ is a fix point of $G_U$.
	Let $x=(x_1,\ldots,x_{2n+1})$ be local coordinates of $\widetilde U$
	with $x(x_0)=0$. 
	For every $j\in\mathbb N$, take $k_j\in\mathbb N$ with $k_j\geq k_0$ 
	such that 
	\begin{equation}\label{e-gue220806yyd}
		\abs{x_{j+k_j}}\leq\frac{1}{j^3}\quad \text{ and }\quad \abs{y_{j+k_j}}\leq\frac{1}{j^3}\,.
	\end{equation}
	From Lemma~\ref{l-gue220806yyd}, we see that $(x_{j+k_j}, y_{j+k_j})\subset A_j$. Thus, for every $j$, we replace $(x_j,y_j)$ by $(x_{j+k_j},y_{j+k_j})$ and conclude that there are $x_k,\,y_k\in M$, $x_k\neq y_k$ such that $\Phi_{kp}(x_k)=\Phi_{kp}(y_k)$, for each $k$ and 
	\begin{equation}\label{e-gue220806ycdb}
		\lim_{k\To+\infty}\abs{k^2x_k}=\lim_{k\To+\infty}\abs{k^2y_k}=0\,.
	\end{equation}
	Put
	\[f_k(t)=\frac{\lvert \Pi^{(0)}_{kp}(tx_k+(1-t)y_k,y_k)\rvert^2}{\Pi^{(0)}_{kp}(tx_k+(1-t)y_k)\Pi^{(0)}_{kp}(y_k)}   \]
	where we write $\Pi^{(0)}_{kp}(x)=\Pi^{(0)}_{kp}(x,x)$, for any $x\in X$. From 
	\eqref{e-gue220806ycdb}, we can check that 
	\begin{equation}f_k(t)=\sum_{g\in G_U}e^{-2kp\,\mathrm{Im}\Psi(t{x}_k+(1-t){y}_k,g\cdot{y}_k)}\cdot\left[1+\frac{1}{k}\tilde{R}_k(t)+\varepsilon_k(t)\right], \label{eq:aa}
	\end{equation}
	where $$\abs{\pr^j_t\widetilde{R}_k(t)}\leq C_j\abs{x_k-y_k}^j\quad \text{and}\quad \abs{\pr^j_t\varepsilon_k(t)}\leq C_{N,j}k^{-N}\abs{x_k-y_k}^j\,,$$ for all $j\in\mathbb N\cup\set{0}$ and every $N\in\mathbb N$, where $C_j, C_{N,j}>0$ are constants independent of $k$. For ease of notation, pose $F_k(t):=-2kp\,\mathrm{Im}\Psi(t{x}_k+(1-t){y}_k,g\cdot {y}_k)$ for given $x_k$, $y_k$ and $g\in G_U$. 
	From \eqref{e-gue220806ycdb}, we have
	\[\lvert F'_k(t) \rvert=\lvert\langle -2kp\,\mathrm{Im}\Psi'_x(t {x}_k+(1-t) {y}_k,g\cdot {y}_k),  {x}_k-{y}_k \rangle\rvert \leq \frac{1}{c_0k}\lvert {x}_k-{y}_k\rvert  \]
	and
	\[F''_k(t)=\langle -2kp\,\mathrm{Im}\Psi''_x(t{x}_k+(1-t){y}_k,g\cdot {y}_k), x_k-y_k \rangle < -c_0k\,\lvert x_k- y_k\rvert^2\]
	for the computation of the second derivative of \eqref{eq:aa}, here $c_0$ is a positive constant. We get
	\begin{equation}\limsup_k\frac{f''(t)}{k\,\lvert x_k-y_k\rvert^2}< 0  \label{eq:lim}
	\end{equation}
	for each $t$.
	
	By Cauchy-Schwartz inequality we have $0\leq f_k(t)\leq 1$, for any $t\in [0,1]$. From $\Phi_{kp}(x_k)=\Phi_{kp}(y_k)$, for each $k$, we can check that $f_k(0)=f_k(1)=1$. Thus, for each $k$, there is a $t_k\in [0, 1]$ such that $f''(t_k)=0$. Hence,
	\[\limsup_k\frac{f''(t)}{k\,\lvert x_k-y_k\rvert^2}\geq 0 \]
	which contradicts \eqref{eq:lim}.

	\subsection{Quantization commutes with reduction}
	\label{sec:app}
	
	In this section, we assume that $X$ admits a CR compact Lie group action $G$, ${\rm dim\,}G=d$. 
	We will work Assumption~\ref{a-gue170123I} and suppose that $X$ is strongly pseudoconvex. We will use the same notations as in Section~\ref{s-gue220801yyd}. The moment map associated to $\omega_0$ is the map given by $\mu: X\To \mathfrak{g}^*$ such that for all $x\in X$ and $\xi\in\mathfrak{g}$, we have
	\begin{equation}\label{e-gue220801yydp}
		\langle\,\mu(x)\,,\,\xi\,\rangle=\omega_0(x)(\xi_X).
	\end{equation}
	Let us recall the following theorem which is well-known in the setting of symplectic manifolds. Throughout this section we will always assume that $0$ is a regular value, and thus $\mu^{-1}(0)$ is a suborbifold of $X$. 
	
	\begin{prop}
		Suppose that $0$ is a regular value, then the action of $G$ on $\mu^{-1}(0)$ is locally free.
	\end{prop}
	\begin{proof}
		Since $0$ is a regular value then for each $p\in \mu^{-1}(0)$, $\mathrm{d}_p\mu:T_pX \rightarrow \mathfrak{g}^{*}$ is surjective. Let us identify $\mathfrak{g}\cong \mathfrak{g}^*$, then for every $\xi \in \mathfrak{g}\setminus \{0\}$ there exists $V\in T_pX$ such that $\mathrm{d}_p\mu(V)=\xi$. Thus, we have
		\[0\neq \langle\mathrm{d}_p\mu(V),\xi \rangle = \mathrm{d}_p\omega_0(\xi_X(p),V)=-b_p(\xi_X(p),J_pV) \]
		and we can conclude that $\xi_X(p)\neq 0$.
	\end{proof}
	
	\begin{prop}
		Suppose $0$ is a regular value, then $X_G=\mu^{-1}(0)/G$ is an orbifold.
	\end{prop}
	\begin{proof}
		By the previous proposition the action of $G$ on $\mu^{-1}(0)$ is locally free. From the proof of Corollary B$.31$ in \cite{ggk} we can conclude.
	\end{proof}
	
	Corollary B$.31$ in \cite{ggk} is a consequence of the Slice Theorem which states that there exists a $G$-equivariant diffeomorphism $\phi$ between $G\times_{G_x} D$ and a tubular neighborhood of the orbit $G\cdot x$, we refer to Appendix B in \cite{ggk} for precise definitions. More precisely, $\phi$ is induced by the exponential map $E_p:U_0\rightarrow U\subseteq M$, where $U_0$ is an open neighborhood of $0$ in $T_xM$. We have $G_x$ equivariant decomposition 
	\[T_xX = T_x(G\cdot x)\oplus W \] 
	where $W$ is the normal bundle to $G\cdot x$. As a consequence of this, it is easy to see the the CR orbifold structure on $X$ descends naturally to a CR orbifold structure on $X$ and so the other $G$-invariant structures defined on $X$. Moreover, we can repeat the proof of~\cite[Theorem 2.5]{HMM} and deduce that $X_G$ is a strongly pseudoconvex CR orbifold. 
	
	We take the Hermitian metric $\langle\,\cdot\,|\,\cdot\,\rangle$ on $\mathbb CTX$ so that 
	$\langle\,\cdot\,|\,\cdot\,\rangle$ is $G$-invariant. Let $\langle\,\cdot\,|\,\cdot\,\rangle_{X_G}$ be the Hermitian metric on $\mathbb CTX_G$ induced by $\langle\,\cdot\,|\,\cdot\,\rangle$ and let $(\,\cdot\,|\,\cdot\,)_{X_G}$ be the inner product 
	on $L^2(X_G)$ induced by $\langle\,\cdot\,|\,\cdot\,\rangle_{X_G}$.
	Let $\Box^{(0)}_{b,X_G}$ be the Kohn Laplacian on the orbifold $X_G$ with respect to $(\,\cdot\,|\,\cdot\,)_{X_G}$. 
	Let $\Pi_{X_G}: L^2(X_G)\To\Box^{(0)}_{b,X_G}$ be the Szeg\H{o} projection.
	Fix $p\in\mu^{-1}(0)$. Let $(\widetilde U,G_U)\To U$ be an orbifold chart of $X$, $p\in U$.
	By the previous discussion  we have that $G\cdot p= [p]\in X_G$ and a sufficiently small neighborhood $U_{[p]}$ of $[p]$ in $X_G$ has a orbifold local chart $(G_p\times G_U,\widetilde W)\rightarrow U_{[p]}$, where $\widetilde W/G_p=(\widetilde U\cap\mu^{-1}(0))/G_p$ and $G_p:=\set{g\in G;\, gp=p}$. Thus, from Theorem~\ref{thm:main}, the Szeg\H{o} kernel for the orbifold $X_G$ has the following expression in the local chart $(G_p\times G_U,\widetilde W)\rightarrow U_{[p]}$ 
	\begin{equation} \label{eq:szegoorbi}
		\Pi_{X_G}(x,y)\equiv\sum_{h\in G_U, g\in G_p}\int _0^{+\infty} e^{it\, \varphi_{-}^{X_G}(\widetilde{x},g\cdot h\cdot \widetilde{y})} s_{-}^{X_G}(\widetilde{x},g\cdot h\cdot \widetilde{y},t)\,\mathrm{d}t 
	\end{equation}
	where $x,y\in U_{[p]}$ and $ \varphi_{-}^{X_G}$ and $s_{-}^{X_G}$ are respectively the complex phase function and the symbol of $S_{X_G}$ as described in Theorem \ref{thm:main} for the CR orbifold $X_G$.
	
	Let use recall that $$({\rm Ker\,}\Box^{(0)}_b)^G:=\set{u\in{\rm Ker\,}\Box^{(0)}_b;\, g^*u=u, \forall g\in G}\,.$$ 
	Let $\Pi_G: L^2(X)\To ({\rm Ker\,}\Box^{(0)}_b)^G$ be the orthogonal projection ($G$-invariant 
	Szeg\H{o} projection).
	In \cite{hsiaohuang} and \cite{gh}, we studied the $G$-invariant Szeg\H{o} kernel. 
	
	We pause and introduce some notations. Fix $x\in\mu^{-1}(0)$, consider the map
	\[\begin{split}
		R_x: \mathfrak{g}^*&\To\mathfrak{g}^*,\\
		u&\To R_xu,\  \langle\,R_xu\,|\,v\,\rangle=\langle\,d\omega_0(x)\,,\,Ju\wedge v\,\rangle,\ \ v\in\mathfrak{g}^*,
	\end{split}\]
	where $J$ is the natural complex structure map on $HX:={\rm Re\,}T^{1,0}X$. 
	Let 
	\[{\rm det\,}R_x=\lambda_1(x)\cdots\lambda_d(x),\]
	$\lambda_j(x)$, $j=1,\ldots,d$, are the eigenvaules of $R_x$ with respect to $\langle\,\cdot\,|\,\cdot\,\rangle$. Put $$Y_x:=\set{gx\in\mu^{-1}(0);\,g\in G}$$ and let $\mathrm{dV}_{Y_x}$ be the volume form on $Y_x$ induced by the given Hermitian metric $\langle\,\cdot\,|\,\cdot\,\rangle$. Put
	\[V_{{\rm eff\,}}:=\int_{Y_x}\mathrm{dV}_{Y_x}.\]
	
	We now come back to our situation. From Theorem~\ref{thm:main}, we can 
	repeat the procedure in~\cite{gh} and deduce the following: Fix $p\in\mu^{-1}(0)$. Let $(\widetilde U,G_U)\To U$ be an orbifold chart of $X$, $p\in U$. Then, 
	the distributional kernel of $\Pi_G$ satisfies
	\begin{equation}\label{e-gue210802yyd}
		\Pi_{G}(x,y)\equiv\sum_{h\in G_U, g\in G_p}\int_0^{\infty} e^{\imath t\,\Phi_-(\widetilde x,g\cdot h\cdot\widetilde y)}a_{-}(\widetilde x,\,g\cdot h\cdot\widetilde y,\,t)\,\mathrm{d}t\ \ \ \mbox{on $U\times U$},
	\end{equation}
	where 
	\[a_{-}(\widetilde x,\widetilde y,\,t)\sim \sum_{j=0}^{+\infty} a_{-}^{j}(\widetilde x,\widetilde y)\,t^{n-d/2-j}\in S^{n-d/2}_{1,0}(\widetilde U\times\widetilde U\times \mathbb{R})\,,\]
	$a_{-}^{j}\in\mathcal{C}^{\infty}(\widetilde U\times \widetilde U),\, j\in\mathbb N_0$, and for every $\widetilde x\in\widetilde{\mu^{-1}(0)}$, $\widetilde{\mu^{-1}(0)}\subset\widetilde U$ is the lifting of $\mu^{-1}(0)$ on $\widetilde U$, we have 
	\begin{equation}\label{e-gue210802yydI}
		a_{-}^{0}(\widetilde x,\,\widetilde x)=2^{d-1}\frac{1}{V_{{\rm eff\,}}(x)|G_x|}\pi^{-n-1+\frac{d}{2}}\abs{\det R_{\widetilde x}}^{-\frac{1}{2}}\abs{\det\mathcal{L}_{\widetilde x}}\,. \end{equation}
	We refer to \cite{hsiaohuang} or \cite{gh} for a precise description of the complex phase function $\Phi_-$ and the symbol $a_-$. Notice that $x$ and $y$ can be assumed to be $x=\exp_p(v)$ and $y=\exp_p(w)$ for some $v,\,w \in W$ since $S^G_-$ is $G$-invariant. 
	
	Now we can introduce explicitly a map $\sigma: ({\rm Ker\,}\Box^{(0)}_b)^G\rightarrow {\rm Ker\,}\Box^{(0)}_{b,X_G}$. 
	Let $\mathcal{C}^\infty(\mu^{-1}(0))^G$ denote the set of all $G$-invariant smooth functions on $\mu^{-1}(0)$. Let 
	\[
	\iota_G:\mathcal{C}^\infty(\mu^{-1}(0))^G\To\mathcal{C}^{\infty}(X_G)
	\]
	be the natural identification. Let $\iota:\mu^{-1}(0)\To X$ be the natural inclusion and let $\iota^*:\mathcal{C}^{\infty}(X)\To\mathcal{C}^{\infty}(\mu^{-1}(0))$ be the pull-back of $\iota$. 
	Let 
	\[
	f(x)=\sqrt{V_{{\rm eff\,}}(x)\abs{G_x}}.
	\]
	Eventually, let $E\in L^{-\frac{d}{4}}_{{\rm cl\,}}(X_G)$ be an elliptic pseudodifferential operator with principal symbol 
	$\sigma^0_E(x,\xi)=\abs{\xi}^{-\frac{d}{4}}$, we define
	\begin{equation}\label{e-gue180308II}
		\begin{split}
			\sigma: \mathcal{C}^{\infty}(X)&\To {\rm Ker\,}\Box^{(0)}_{b,X_G},
			\\
			u&\mapsto {\Pi_{X_G}}\circ E\circ\iota_G\circ  f\circ \iota^*\circ\Pi_Gu.
		\end{split}
	\end{equation}
	Let $\sigma^{\,*}: \mathcal{C}^\infty(X_{G})\To \mathcal{D}'(X)$ be
	the formal adjoint of $\sigma$. Following the same computations as in \cite{hsiaohuang} with some modification, by an application of stationary phase lemma of \cite{ms} and by making use of \eqref{eq:szegoorbi}, which is a consequence of Theorem~\ref{thm:main} and by equation \eqref{e-gue210802yyd}, we obtain
	
	\begin{thm}\label{t-gue211213yyd}
		Under the hypothesis of this section, suppose that $\Box^{(0)}_{b,X_G}$ has closed range. Then, $\sigma^{\,*}$ maps $\mathcal{C}^{\infty}(X_{G})$ continuously to $\mathcal{C}^{\infty}(X)$. Hence, $\sigma^*\,\sigma: \mathcal{C}^{\infty}(X)\To \mathcal{C}^{\infty}(X)$ is a well-defined continuous operator whose distribution kernel satisfies 
		\begin{equation}\label{e-gue211213yydp}
			\sigma^*\,\sigma=\Pi_G+\Gamma,
		\end{equation}
		where $\Gamma: \mathcal{C}^{\infty}(X)\To\mathcal{C}^{\infty}(X)$ is a continuous map and satisfies the following: there exists $\varepsilon>0$ such that $\Gamma: H^s(X)\To H^{s+\varepsilon}(X)$ is continuous for every $s\in\mathbb\mathbb R$, $H^s(X)$ is the Sobolev space of functions over $X$ order $s$. 
		
		Moreover, 
		\begin{equation}\label{e-gue220803yyd}
			\sigma\,\sigma^*=\Pi_{X_G}+\hat\Gamma,
		\end{equation}
		where $\hat\Gamma: \mathcal{C}^{\infty}(X_G)\To\mathcal{C}^{\infty}(X_G)$ is a continuous map and satisfies the following: there exists $\delta>0$ such that $\hat\Gamma: H^s(X_G)\To H^{s+\delta}(X_G)$ is continuous for every $s\in\mathbb\mathbb R$, $H^s(X_G)$ is the Sobolev space of functions over $X_G$ order $s$. 
	\end{thm} 
	
	From Theorem~\ref{t-gue211213yyd}, we see that $\sigma$ can be extended by density to
	\begin{equation}\label{e-gue211213yydr}
		\sigma: L^2(X)\To {\rm Ker\,}\Box^{(0)}_{b,X_G}.
	\end{equation} 
	
	\begin{cor}\label{c-gue220802yyd}
		Under the preceding assumptions and notations, the map $\sigma$  given by \eqref{e-gue180308II} and \eqref{e-gue211213yydr} has the following properties:
		\begin{itemize}
			\item[(i)] ${\rm Ker\,}\sigma\cap({\rm Ker\,}\Box^{(0)}_{b})^G$ is a finite dimensional subspace of $\mathcal{C}^{\infty}(X)$. 
			\item[(ii)] ${\rm Coker\,}\sigma\cap {\rm Ker\,}\Box^{(0)}_{b,X_G}$ is a finite dimensional subspace of $\mathcal{C}^{\infty}(X_G)$.
		\end{itemize}
		Moreover, we have 
		$\sigma: ({\rm Ker\,}\Box^{(0)}_{b})^G\To {\rm Ker\,}\Box^{(0)}_{b,X_G}$ is a Fredholm map.
	\end{cor}
	
	\begin{proof}
		Let $u\in{\rm Ker\,}\sigma\cap({\rm Ker\,}\Box^{(0)}_{b})^G$. From \eqref{e-gue211213yydp}, we have 
		\begin{equation}\label{e-gue211213ycdt}
			0=\sigma^*\,\sigma u=(\Pi_G+\Gamma)u=u+\Gamma u.
		\end{equation}
		From the previous theorem and \eqref{e-gue211213ycdt}, we deduce that $u\in H^{\varepsilon}(X)$. By applying again the previous theorem and \eqref{e-gue211213ycdt}, we get $u\in H^{2\varepsilon}(X)$. Continuing in this way, we conclude that $u\in\mathcal{C}^{\infty}(X)$. 
		
		Suppose that ${\rm Ker\,}\sigma\cap({\rm Ker\,}\Box^{(0)}_{b})^G$ is not a finite dimensional subspace of $\mathcal{C}^{\infty}(X)$. We can find $u_j\in{\rm Ker\,}\sigma\cap({\rm Ker\,}\Box^{(0)}_{b})^G$, $j=1,2,\ldots$, $(\,u_j\,|\,u_k\,)=\delta_{j,k}$, for every $j, k=1,\ldots$. From the previous theorem and \eqref{e-gue211213ycdt}, 
		we notice that $\{u_j\}_{j\in\mathbb{N}}$ is a bounded set in $H^{\varepsilon}(X)$. By Rellich's lemma, there is a subsequence $1\leq j_1<j_2<\cdots$, $\lim_{s\To+\infty}j_s=+\infty$, such that $u_{j_s}\To u$ in $L^2(X)$ as $s\To+\infty$, for some $u\in L^2(X)$. Since  
		$(\,u_j\,|\,u_k\,)=\delta_{j,k}$, for every $j, k=1,\ldots$, we have
		\[\lVert u_{j_s} - u_{k_s} \rVert_{L^2}=\sqrt{2} \]
		and we get a contradiction since $u_{j_s}$ and $u_{k_s}$ converge to $u$ as $k_s,j_s$ goes to infinity. Thus, ${\rm Ker\,}\sigma\cap({\rm Ker\,}\Box^{(0)}_{b})^G$ is a finite dimensional subspace of $\mathcal{C}^{\infty}(X)$. 
		
		We can repeat the procedure above with minor change and deduce that ${\rm Coker\,}\sigma\cap {\rm Ker\,}\Box^{(0)}_{b,X_G}$ is a finite dimensional subspace of $\mathcal{C}^{\infty}(X_G)$.
	\end{proof} 
	
	Recall that when  $X$ is a circle orbibundle, we have a natural circle action which commutes with the action of $G$ on $X$ and it is transversal to the CR structure. This induces a decomposition of $({\rm Ker\,}\Box^{(0)}_{b})^G$ into Fourier components $({\rm Ker\,}\Box^{(0)}_{b})^G_k$ and similarly ${\rm Ker\,}\Box^{(0)}_{b,X_G}$ splits as a direct sum $({\rm Ker\,}\Box^{(0)}_{b,X_G})_k$, $k\in \mathbb{Z}$. Now, we can take the limits as $k$ goes to infinity in the previous corollary and we get that quantization commutes with reduction for $k$ large, $X$ and $X_G$ may be orbifolds.
	
	\bigskip
	\textbf{Acknowledgments:} This project was started during the first author’s postdoctoral fellowship at the National Center for Theoretical Sciences in Taiwan; we thank the Center for the support. Andrea Galasso thanks Yi-Sheng Wang for conversations about orbifolds. Chin-Yu Hsiao was partially supported by Taiwan Ministry of Science and Technology projects  108-2115-M-001-012-MY5, 109-2923-M-001-010-MY4.

\end{document}